\documentclass[a4,12pt]{amsart}
\usepackage{amssymb,amstext,amsmath,amscd,amsthm,amsfonts,enumerate,latexsym}
\usepackage{color}
\usepackage[dvipdfmx]{graphicx}
\usepackage[all]{xy}

\theoremstyle{plain}
\newtheorem{thm}{Theorem}[section]

\newtheorem*{thm*}{Theorem}
\newtheorem*{cor*}{Corollary}

\newtheorem{prop}[thm]{Proposition}

\newtheorem{lemma}[thm]{Lemma}
\newtheorem{lem}[thm]{Lemma}
\newtheorem{cor}[thm]{Corollary}

\newtheorem{claim}{Claim}
\newtheorem*{claim*}{Claim}

\theoremstyle{definition}
\newtheorem{defn}[thm]{Definition}

\newtheorem{ex}[thm]{Example}

\newtheorem{rem}[thm]{Remark}

\newtheorem{conj}[thm]{Conjecture}

\newtheorem*{acknowledgments}{Acknowledgments}
\newtheorem*{observation}{Observation}

\theoremstyle{remark}

\newtheorem*{proof of claim}{{\sl Proof of Claim}}

\numberwithin{equation}{thm}

\def\pd{\operatorname{pd}}

\def\Ext{\operatorname{Ext}}

\def\Im{\operatorname{Im}}

\def\Ker{\operatorname{Ker}}

\def\Hom{\operatorname{Hom}}

\def\Tor{\operatorname{Tor}}
\def\Max{\operatorname{Max}}

\def\Ker{\mathrm{Ker}}
\def\Im{\mathrm{Im}}

\def\rank{\mathrm{rank}}

\newcommand{\rmr}{\mathrm{r}}

\newcommand{\rmv}{\mathrm{v}}

\newcommand{\fka}{\mathfrak{a}}

\newcommand{\fkm}{\mathfrak{m}}
\newcommand{\fkn}{\mathfrak{n}}

\newcommand{\fkp}{\mathfrak{p}}
\newcommand{\fkq}{\mathfrak{q}}

\newcommand{\fkM}{\mathfrak{M}}
\newcommand{\fkN}{\mathfrak{N}}

\newcommand{\mapright}[1]{%
\smash{\mathop{%
\hbox to 1cm{\rightarrowfill}}\limits^{#1}}}

\newcommand{\mapleft}[1]{%
\smash{\mathop{%
\hbox to 1cm{\leftarrowfill}}\limits_{#1}}}

\def\height{\mathrm{ht}}

\def\ol{\overline}

\title[The Auslander-Reiten conjecture for certain non-Gorenstein CM rings]{The Auslander-Reiten conjecture for certain non-Gorenstein Cohen-Macaulay rings}
\author{Shinya Kumashiro}
\address{National Institute of Technology, Oyama College, 771 Nakakuki, Oyama, Tochigi, 323-0806, Japan}	
\email{skumashiro@oyama-ct.ac.jp}
\thanks{The author was supported by JSPS KAKENHI Grant Numbers JP21K13766 and JP19J10579 and Overseas Study Grant for Graduate Students in the Frontier Science Program of Chiba University.}

\begin{document}

\begin{abstract}
The Auslander-Reiten conjecture is a notorious open problem about the vanishing of Ext modules. In a Cohen-Macaulay complete local ring $R$ with a parameter ideal $Q$, the Auslander-Reiten conjecture holds for $R$ if and only if it holds for the residue ring $R/Q$. 
In the former part of this paper, we study the Auslander-Reiten conjecture for the ring $R/Q^\ell$ in connection with that for $R$, and prove the equivalence of them for the case where $R$ is Gorenstein and $\ell\le \dim R$. 
In the latter part, we generalize the result of the minimal multiplicity by J. Sally. Due to these two of our results, we see that the Auslander-Reiten conjecture holds if there exists an Ulrich ideal whose residue ring is a complete intersection. We also explore the Auslander-Reiten conjecture for determinantal rings.
\end{abstract}

\maketitle

\section{Introduction}\label{section1}
The Auslander-Reiten conjecture and several related conjectures are problems about the vanishing of Ext modules. As is well-known, the vanishing of cohomology plays a very important role in the study of rings and modules. For a guide to these conjectures, one can consult \cite[Appendix A]{CH} and \cite{CT, Ha, W, W2}. 
These conjectures originate from the representation theory of algebras, and the theory of commutative ring  also greatly contributes to the development of the Auslander-Reiten conjecture; see, for examples, \cite{Ara, HL, HSV, JS}.
Let us recall the Auslander-Reiten conjecture over a commutative Noetherian ring $R$.

\begin{conj}(\cite{AR})\label{a1.1}
Let $M$ be a finitely generated $R$-module. 
If 
%\begin{center}
$\Ext_R^i(M, M\oplus R)=0$ for all $i>0$,
%\end{center}
then $M$ is a projective $R$-module.
\end{conj}

The conjecture is trivially reduced to the local case, thus let $(R, \fkm)$ be a commutative Noetherian local ring and $M$ a finitely generated $R$-module. One can also see that if the Auslander-Reiten conjecture holds for the completion of $R$, then it holds for $R$ (Remark \ref{remrem23}). Thus, we suppose that $R$ is complete. 
Then, a fundamental result concerning the Auslander-Reiten conjecture is that it remains true after quotienting by a non-zerodivisor, that is, the Auslander-Reiten conjecture holds for $R$ if and only if it holds for $R/xR$, where $x\in \fkm$ is a non-zerodivisor of $R$ (Proposition \ref{e2.1}). It is also known that the Auslander-Reiten conjecture holds for all complete intersections (see \cite[4.2 Theorem]{AB} or this is listed in \cite[Appendix A]{CH}). After that, Huneke, Leuschke, and Araya proved that the conjecture holds for Gorenstein rings which are complete intersections in codimension one (\cite[Theorem 3.]{Ara} and \cite[Theorem 0.1]{HL}). However, it is not known whether the Auslander-Reiten conjecture holds for arbitrary Gorenstein rings.

The conjecture for non-Gorenstein rings is more mysterious. If $R$ is a Gorenstein ring, then the assumption that $\Ext_R^i (M, R)=0$ for all $i>0$ in the conjecture just requires $M$ is a maximal Cohen-Macaulay module. However, in some classes of non-Gorenstein rings, the assumption $\Ext_R^i (M, R)=0$ for all $i>0$ forces $M$ to be free.
Let us note several results on non-Gorenstein rings. Huneke and Leuschke also showed that the Auslander-Reiten conjecture holds for Cohen-Macaulay $\mathbb{Q}$-algebras which are complete intersections in codimension one. Goto, Takahashi, and Taniguchi proved that $\Ext_R^i (M, R)=0$ for all $i>0$ implies $M$ is free for non-Gorenstein almost Gorenstein rings in the sense of \cite{GTT} (\cite[Theorem 0.1]{HL} and \cite[proof of Corollary 4.5.]{GTT}). 

Now let us explain the purpose of this paper. In this paper, we compare $R$ and $R/Q^\ell$ in the study of the Auslander-Reiten conjecture, where $Q$ is an ideal of $R$ generated by a regular sequence on $R$ (Theorem \ref{a2.1}).
It is well-known that the Auslander-Reiten conjecture holds for $R$ if and only if it holds for $R/Q^\ell$ if $\ell=1$. However, $R/Q^\ell$ is neither Gorenstein nor reduced if $\ell>1$, thus Theorem \ref{a2.1} provides new classes of rings where the Auslander-Reiten conjecture holds. Here let us note one application from Theorem \ref{a2.1} to determinantal rings, which is an important class of rings.

Let $s$, $t$ be positive integers and  $A[\mathbf{X}]=A[X_{ij}]_{1\le i\le s, 1\le j\le t}$ a polynomial ring over a commutative ring $A$. 
Assume $s\le t$ and let $\mathbb{I}_s(\mathbf{X})$ 
denote the ideal of $A[\mathbf{X}]$ generated by the $s\times s$ minors of the $s\times t$ matrix $\left(X_{ij}\right)$. Then $A[\mathbf{X}]/\mathbb{I}_s(\mathbf{X})$ is called a {\it determinantal ring} of $A$. With these assumptions and notations, we have the following.

\begin{thm}{\rm (Theorem \ref{f2.7})}
Suppose that $A$ is a complete intersection for all localizations at maximal ideals. Then the Auslander-Reiten conjecture holds for the determinantal ring $A[\mathbf{X}]/\mathbb{I}_s(\mathbf{X})$ if $2s\le t+1$.
\end{thm}

In the latter part of this paper, we study a new class of rings arising from Theorem \ref{a2.1}, that is, the class of rings $R$ that $R/\fkq\cong S/Q^2$ for some parameter ideal $\fkq$ of $R$, complete intersection $S$, and parameter ideal $Q$ of $S$.
We will see that the condition is characterized by an ideal condition, and the condition is strongly related to the existence of certain Ulrich ideals. Here the notion of Ulrich ideals given by \cite{GOTWY} is a generalization of maximal ideals of rings possessing minimal multiplicity. 
It is known that Ulrich ideals enjoy many good properties, see \cite{GOTWY, GTT2} and \cite[Theorem 1.2]{GIK}. The ubiquity and existence of Ulrich ideals are also studied (\cite{GIK, GOTWY2, GTT2}). In the current paper, we study the existence of Ulrich ideals whose residue rings are complete intersections. As a goal of this paper, we have the following. Let $v$ denote the embedding dimension $\mu_R(\fkm)$ of $R$, where $\mu_R(M)$ denotes the number of elements in a minimal system of generators of $M$. We denote by $r$ the Cohen-Macaulay type $\ell_R(\Ext_R^d(R/\fkm, R))$ of $R$, where $\ell_R(M)$ denotes the length of an $R$-module $M$ (\cite[Definition 1.2.15]{BH}).

\begin{thm}{\rm (Corollary \ref{c2.10})}\label{f1.5}
Let $(R, \fkm)$ be a Cohen-Macaulay local ring of dimension $d$. Suppose that there exists an Ulrich ideal whose residue ring is a complete intersection. 
Then the following assertions are true.
\begin{enumerate}[{\rm (1)}] 
\item $r+d\le v$.
\item There exists an isomorphism of rings $R/\fkq\cong S/Q^2$ for some parameter ideal $\fkq$ of $R$, local complete intersection $S$ of dimension $r$, and parameter ideal $Q$ of $S$.
\item The Auslander-Reiten conjecture holds for $R$.
\item Suppose that there are a regular local ring $T$ of dimension $v$ and a surjective ring homomorphism $T\to R$.  Let
$0\to F_{v-d}\to \cdots \to F_1\to F_0 \to R \to 0$
be a minimal $T$-free resolution of $R$. Then
$$
\rank_T F_0=1\ \  \text{and}\ \  \rank_T F_i=\sum_{j=0}^{v-r-d} \beta_{i-j}{\cdot}\binom{v-r-d}{j}
$$
for $1\le i\le v-d$, where $\beta_k=
\begin{cases}
1 & \text{if $k=0$}\\
k{\cdot}\binom{r+1}{k+1} & \text{if $1\le k\le r$}\\
0& \text{otherwise.}
\end{cases}$
\end{enumerate}
\end{thm}

The assertions (1) and (2) of Theorem \ref{f1.5} claim that the existence of Ulrich ideal induces an inequality of invariants of rings, and determines the structure of the ring. Furthermore, Theorem \ref{f1.5} (4) recovers the result of J. Sally \cite[Theorem 1(iii)]{S} by choosing the maximal ideal $\fkm$ as a certain Ulrich ideal.

Let us fix our notations throughout this paper. In what follows, unless otherwise specified, let $R$ denote a Cohen-Macaulay local ring with the maximal ideal $\fkm$. We denote by $\widehat{-}$ the $\fkm$-adic completion. For each finitely generated $R$-module $M$, let $\mu_R (M)$ (resp. $\ell_R (M)$) denote the number of elements in a minimal system of generators of $M$ (resp. the length of $M$). 
If $M$ is a Cohen-Macaulay $R$-module, $\rmr_R (M)$ denotes the Cohen-Macaulay type $\ell_R(\Ext_R^d(R/\fkm, R))$ of $M$ (\cite[Definition 1.2.15]{BH}). Let $\rmv(R)$ (resp. $\rmr(R)$) denote the embedding dimension of $R$ (resp. the Cohen-Macaulay type of $R$). For convenience, letting $M$ and $N$ be $R$-modules, $\Ext_R^{>0} (M, N)=0$ (resp. $\Tor_{>0}^R (M, N)=0$) denotes $\Ext_R^{i} (M, N)=0$ for all $i>0$ (resp. $\Tor_{i}^R (M, N)=0$ for all $i>0$).

%%%%%%%%%%%%%%%%%%%%%%%%%%%%%%%%%%%%%%%%%%%%%%%%%%%%%%%%%%%%
%%%%%%%%%%%%%%%%%%%%%%%%%%%%%%%%%%%%%%%%%%%%%%%%%%%%%%%%%%%%

\section{powers of parameter ideals and determinantal rings}\label{section2}

The purpose of this section is to prove Theorem \ref{a2.1} and apply it to determinantal rings. 
First of all, let us sketch a proof of the fact that quotienting by a non-zerodivisor preserves the truth of the Auslander-Reiten conjecture (Proposition \ref{e2.1}). Indeed, Theorem \ref{a2.1} is based on this fact.

\begin{lem}\label{nn21} {\rm (\cite[Lemma 1.3.5]{BH})}
Let $(R, \fkm)$ be a Noetherian local ring and $M$ be a finitely generated $R$-module. 
Let $x\in \fkm$ be a non-zerodivisor of $R$ and $M$. Let $\ol{*}$ denote $R/(x)\otimes_R *$. Then, $\pd_R M=\pd_{\ol{R}} \ol{M}$.
\end{lem}

%\begin{proof}
%Let $F_{\bullet} \to M \to 0$ be a minimal $R$-free resolution of $M$. Then, $\ol{F}_{\bullet} \to \ol{M} \to 0$ is an $\ol{R}$-free resolution of $\ol{M}$. Furthermore, this is minimal since all entries of matrices representing $F_{i+1} \to F_i$ are in $\ol{\fkm}$. It follows that for a non-negative integer $n$, $\pd_R M\le n$ if and only if $F_{n+1}= 0$  if and only if $\pd_{\ol{R}} \ol{M}\le n$.
%\end{proof}

\begin{prop}\label{e2.1} {\rm (\cite[Theorem 4.5(1)]{CT}, \cite[Lemma 1.1]{HL})}
Let $(R, \fkm)$ be a Noetherian local ring and $x\in \fkm$ be a non-zerodivisor of $R$. Consider the following conditions.
\begin{enumerate}[\rm(1)] 
\item The Auslander-Reiten conjecture holds for $R$.
\item The Auslander-Reiten conjecture holds for $R/xR$.
\end{enumerate}
Then {\rm (2) $\Rightarrow$ (1)} holds. {\rm (1)  $\Rightarrow$ (2)} also holds if $R$ is complete.
\end{prop}

\begin{proof}
(2) $\Rightarrow$ (1): Let $M$ be a finitely generated $R$-module such that $\Ext_R^{>0}(M, M\oplus R)=0$. 
Take an exact sequence $0\to X\to F\to M\to 0$, where $F$ is a free $R$-module of rank $\mu_R (M)$. By applying the functors 
\begin{center}
$\Hom_R (-, R)$,\quad $\Hom_R (M, -)$, and\quad $\Hom_R (-, X)$
\end{center}
to the above short exact sequence, we have 
\begin{center}
$\Ext_R^{>0}(X, R)=0$,\quad $\Ext_R^{i}(M, M) \cong \Ext_R^{i+1}(M, X)$, and\quad $\Ext_R^{i}(X, X) \cong \Ext_R^{i+1}(M, X)$
\end{center}
 for all $i\ge 1$. Thus, $\Ext_R^{>0}(X, X\oplus R)=0$. Note that $x$ is a non-zerodivisor of $X$ because of an exact sequence $0\to X\to F$. Hence, by applying the functor $\Hom_R(-, X\oplus R)$ to the short exact sequence $0 \to X\xrightarrow{ x} X \to \ol{X} \to 0$, where $\ol{*}$ denotes $R/(x)\otimes_R *$, we obtain that $\Ext_R^i(\ol{X}, X\oplus R)=0$ for all $i\ge 2$. By the result of Rees (\cite[Lemma 3.1.16]{BH}), it follows that $\Ext_{\ol{R}}^{>0}(\ol{X}, \ol{X}\oplus \ol{R})=0$.
Therefore, the $\ol{R}$-module $\ol{X}$ is $\ol{R}$-free  and the $R$-module $X$ is also $R$-free by Lemma \ref{nn21}. We then again apply the $R$-dual $(-)^*$ to the short exact sequence $0\to X\to F\to M\to 0$. Thus, we get $0 \to M^* \to F^* \to X^* \to 0$. Since $X$ is $R$-free, this exact sequence splits; hence, $M^*$ is $R$-free. 
By applying the $R$-dual to the above exact sequence, we have the exact sequence $0\to X^{**}\to F^{**}\to M^{**}\to 0$.
Therefore, we have a commutative diagram
\[
\xymatrix{
0 \ar[r] & X\ar[r] \ar[d]_{\varphi_X}  &F \ar[r] \ar[d]_{\varphi_F} & M\ar[r] \ar[d]_{\varphi_M} & 0\\
0 \ar[r] & X^{**}\ar[r]  & F^{**} \ar[r] & M^{**} \ar[r] & 0, \\
}
\]
where $\varphi_N: N\to N^{**}$ denotes the canonical homomorphism defined by $y\mapsto (N^*\to~R;f\mapsto f(y))$. By the freeness of $X$ and $F$, $\varphi_X$ and $\varphi_F$ are isomorphisms. It follows that $\varphi_M$ is an isomorphism by the snake lemma. Thus, $M\cong M^{**}$ is $R$-free since $M^*$ is $R$-free.

(2)  $\Rightarrow$ (1): Let $N$ be a finitely generated $\ol{R}$-module such that $\Ext_{\ol{R}}^{>0}(N, N\oplus \ol{R})=0$. Then there exists a finitely generated $R$-module $M$ such that $M/xM\cong N$ and $\Tor_{>0}^R (M, \ol{R})=0$ by \cite[Proposition 1.7.]{ADS} since $R$ is complete. By applying the functor $M\otimes_R -$ to the exact sequence $0 \to R \xrightarrow{ x} R\to \ol{R} \to 0$, we have an exact sequence 
\[
\Tor_1^R(M, \ol{R})=0 \to M \xrightarrow{ x} M \to \ol{M} \to 0,
\] 
that is, $x$ is a non-zerodivisor of $M$. Hence, by applying the functor $\Hom_R(-, M\oplus R)$ to the exact sequence $0\to M\xrightarrow{x} M\to N\to 0$, we have 
\[
\Ext_R^i(M, M\oplus R) \xrightarrow{x} \Ext_R^i(M, M\oplus R) \to \Ext_R^{i+1}(N, M\oplus R)
\]
for all $i\ge 1$. By \cite[Lemma 3.1.16]{BH}, it follows that $\Ext_R^{i+1}(N, M\oplus R) \cong \Ext_{\ol{R}}^i(N, N\oplus \ol{R}) =0$. Therefore, by applying Nakayama's lemma to the above exact sequence, we obtain that $\Ext_R^{>0}(M, M\oplus~R)=0$.
Hence, $M$ is free and so is the $\ol{R}$-module $N$ by Lemma \ref{nn21}.
\end{proof}

\begin{rem}\label{remrem23}
Let $R$ be a Noetherian local ring. If the Auslander-Reiten conjecture holds for $\widehat{R}$, then it holds for $R$.
\end{rem}

\begin{proof}
Let $M$ be a finitely generated $R$-module such that $\Ext_R^{>0}(M, M\oplus R)=0$. Then, $\Ext_{\widehat{R}}^{>0}(\widehat{M}, \widehat{M}\oplus \widehat{R})=0$. It follows that $\widehat{M}$ is $\widehat{R}$-free. Thus, $M$ is $R$-free.
\end{proof}

The following is a key to show Theorem \ref{a2.1}.

\begin{lemma}\label{ob1}
Let $(R, \fkm)$ be a Noetherian local ring and $I$ an $\fkm$-primary ideal of $R$. Then we have the following.
\begin{enumerate}[{\rm (1)}] 
\item  Suppose that $R/I$ is a Gorenstein ring. If $\Ext_R^{>0} (M, R/I)=0$, then $\Tor_{>0}^R (M, R/I)=0$.
\item Suppose that $N$ is a finitely generated $R/I$-module. If $\Tor_{>0}^R (M, R/I)=0$, then $\Ext_R^{i} (M, N)\cong\Ext_{R/I}^{i} (M/IM, N)$ for all $i\in \mathbb{Z}$.
\end{enumerate}
\end{lemma}

\begin{proof} 
Let $M$ be a finitely generated $R$-module and $F_{\bullet} \to M \to 0$ a minimal $R$-free resolution of $M$.

(1): Note that
\[
0\to \Hom_R (M, R/I) \to \Hom_R (F_{\bullet}, R/I)  \ \ \text{ and } 
\]
\[
0\to \Hom_{R/I} (M/IM, R/I) \to \Hom_{R/I} (F_{\bullet}/IF_{\bullet}, R/I) 
\]
are isomorphic as complexes. By the assumption that $\Ext_R^{>0} (M, R/I)=0$, the former complex is exact, and therefore the latter complex is also exact. We also note that $R/I$ is Artinian Gorenstein, that is, the injective dimension of $R/I$ is zero. Hence, by applying the $(R/I)$-dual to the latter exact sequence above, $\Hom_{R/I}(\Hom_{R/I}(F_{\bullet}/I F_{\bullet}, R/I), R/I) \to \Hom_{R/I}(\Hom_{R/I}(M/IM, R/I), R/I) \to 0$ is exact. By using the Matlis dual (see for example \cite[Theorem 3.2.13]{BH}), we obtain that $F_{\bullet}/I F_{\bullet} \to M/IM \to 0$ is an exact sequence. It follows that $\Tor_{>0}^R (M, R/I)=0$.

(2): Since
\[
0\to \Hom_R (M, N) \to \Hom_R (F_{\bullet}, N)  \ \ \text{ and } 
\]
\[
0\to \Hom_{R/I} (M/IM, N) \to \Hom_{R/I} (F_{\bullet}/IF_{\bullet}, N) 
\]
are isomorphic as complexes and $F_{\bullet}/I F_{\bullet} \to M/IM \to 0$ is a minimal $R/I$-free resolution of $M/IM$ by the assumption, $\Ext_R^{i} (M, N)\cong\Ext_{R/I}^{i} (M/IM, N)$ for all $i\in \mathbb{Z}$.
\end{proof}

\begin{lem}\label{nnn24}
Let $S$ be a Noetherian ring and $x_1, x_2, \dots, x_n$ be a regular sequence on $S$. Set $Q=(x_1, x_2, \dots, x_n)$. 
Then, $Q^{i-1}/Q^{i}$ is an $S/Q$-free module of rank $\binom{i-1+n-1}{n-1}$ for all $i>0$. 
\end{lem}

\begin{proof}
Let $(S/Q)[X_1, \dots, X_n]$ be the $\mathbb{Z}$-graded polynomial ring over $S/Q$ with $\deg a=0$ for $a\in S/Q$ and  $\deg X_k =1$ for $1 \le k \le n$. For a $\mathbb{Z}$-graded module $N$, we denote by $N_i$ the set of homogeneous elements in $N$ of degree $i$. By \cite[Theorem 1.1.8]{BH}, we then have the isomorphism 
\[
Q^{i-1}/Q^{i} \cong (S/Q)[X_1, \dots, X_n]_{i-1}
\] 
as $S/Q$-modules. It follows that $Q^{i-1}/Q^{i}$ is an $S/Q$-free module and the rank of $Q^{i-1}/Q^{i}$ is the number of monomials of degree $i-1$ for all $i>0$. Hence, we obtain that 
\[
Q^{i-1}/Q^{i} \cong (S/Q)^{\oplus \binom{i-1+n-1}{n-1}}
\] 
for all $i>0$. 
\end{proof}

\begin{thm}\label{a2.1}
Let $(S, \fkn)$ be a Gorenstein complete local ring and $x_1, x_2, \dots, x_n$ be a regular sequence on $S$. 
Set $Q=(x_1, x_2, \dots, x_n)$. 
Then the following conditions are equivalent.
\begin{enumerate}[{\rm (1)}] 
\item The Auslander-Reiten conjecture holds for $S$.
\item The Auslander-Reiten conjecture holds for $S/Q$.
\item There is an integer $\ell>0$ such that the Auslander-Reiten conjecture holds for $S/Q^{\ell}$.
\item For all integers $1\le \ell \le n$, the Auslander-Reiten conjecture holds for $S/Q^{\ell}$.
\end{enumerate}

\end{thm}

\begin{proof}
$(1) \Leftrightarrow (2)$ follows from Proposition \ref{e2.1}, and $(4) \Rightarrow (3)$ is trivial. Hence we have only to show that $(1)\Rightarrow (4)$ and $(3)\Rightarrow (1)$. First of all, we reduce our assertions to the case where $Q$ is a parameter ideal of $S$. Indeed, assume that $n<\dim S$. 
%%%%%%%%%%%%%%%%%%%%%%%%%%%%%%%%%%%%%%%%%%%%%%%%%%%%%%%%%%%%
Set $R=S/Q^{\ell}$. Since $Q^\ell$ is generated by $\ell \times \ell$-minors of the $\ell \times (n+\ell-1)$ matrix 
\begin{equation*}
\left(
\begin{array}{ccccccccc}
x_1 &x_2&x_3&\cdots &x_{n-1}&x_n&&&\\
&x_1&x_2&\cdots&\cdots&x_{n-1} &x_{n}&&\\
&&\ddots&\ddots&\ddots&\ddots &\ddots&\ddots&\\
&&&x_1&x_2&x_{3}&\cdots&x_{n-1} &x_{n}\\
\end{array}
\right),
\end{equation*}
we have $\pd_{S}S/Q^\ell =n$ (see \cite{EN} or \cite[(2.14) Proposition]{BV2}). Hence, $R$ is a Cohen-Macaulay local ring with $\dim R=\dim S-n$ since $Q^\ell$ is perfect in the sense of \cite[Definition 1.4.15]{BH}. 
Thus, if $Q$ is not a parameter ideal of $S$, we have $\dim R>0$. 
Then we can choose $a\in \fkn$ so that $a$ is a non-zerodivisor of $R$ and $S/Q$. It follows that the truth of the Auslander-Reiten conjecture for $S$ and $S/aS$ are equivalent by Proposition \ref{a2.1}. The truth of the Auslander-Reiten conjecture for $R$ and $R/aR$ are also equivalent. Therefore, for each proof of $(1)\Rightarrow (4)$ and $(3)\Rightarrow (1)$, by replacing $R$ and $S$ by $R/aR$ and $S/aS$ respectively and recursively, we finally may assume that $R$ is an Artinian local ring, that is, $Q$ is a parameter ideal of $S$. 

Suppose that $Q$ is a parameter ideal of $S$. We may also assume that $n\ge 2$ and $\ell \ge 2$ by Proposition \ref{e2.1}. 

We consider exact sequences
\begin{align*}
0\to Q^{i-1}/Q^{i}\to S/Q^{i}\to S/Q^{i-1}\to 0 
\end{align*}
of $R$-modules for all $2\le i \le \ell$. By Lemma \ref{nnn24}, we obtain that  $Q^{i-1}/Q^{i} \cong (S/Q)^{\oplus \binom{i-1+n-1}{n-1}}$ as $S/Q$-modules. Therefore, we can rephrase the above exact seuqnces as follows:
\begin{align}\label{nnn242}
0\to (S/Q)^{\oplus \binom{i-1+n-1}{n-1}} \to S/Q^{i}\to S/Q^{i-1}\to 0 
\end{align}
are exact for all $2\le i \le \ell$. Note that $R$ appears in the middle of the exact sequence \eqref{nnn242} for $i=\ell$.

$(1)\Rightarrow (4)$: 
Assume $2\le \ell \le n$. Suppose that $M$ is a finitely generated $R$-module such that $\Ext_R^{>0}(M, M\oplus R)=0$.
We will show that $\Ext_{S/Q}^{>0} (M/QM, M/QM\oplus S/Q)=0$ in several steps.

\begin{claim}\label{claim1}
$\Ext_R^{>0}(M, S/Q)=0$.
\end{claim}

\begin{proof}[Proof of {\rm Claim \ref{claim1}}]
By applying the functor $\Hom_R (M, -)$ to the exact sequences \eqref{nnn242}, we have the following exact sequences
%\begin{small}
\begin{align}\label{xx2}
\begin{split}
\cdots \to &\Ext_R^{j} (M, S/Q)^{\oplus \binom{i-1+n-1}{n-1}} \to \Ext_R^{j} (M, S/Q^{i}) \to \Ext_R^{j} (M, S/Q^{i-1})\\
 \to & \Ext_R^{j+1} (M, S/Q)^{\oplus \binom{i-1+n-1}{n-1}} \to \cdots 
 \end{split}
\end{align}
%\end{small}
of $R$-modules for all $2\le i\le \ell-1$ and $j>0$. For $i=\ell$, since $\Ext_R^{>0}(M, R)=0$, we have isomorphisms
\begin{align}\label{nn244}
\Ext_R^{j} (M, S/Q^{\ell-1})\cong \Ext_R^{j+1} (M, S/Q)^{\oplus \binom{\ell-1+n-1}{n-1}}
\end{align}
for all $j>0$.

Set $E_j=\ell_R (\Ext_R^{j} (M, S/Q))$ for $j>0$. Then, by (\ref{xx2}) and \eqref{nn244}, 
\begin{small}
\[
\begin{split}
E_{j+1}{\cdot}\binom{\ell-1+n-1}{n-1}&=\ell_R (\Ext_R^{j} (M, S/Q^{\ell-1}))\\
&\le E_j{\cdot}\binom{\ell-2+n-1}{n-1} + \ell_R (\Ext_R^{j} (M, S/Q^{\ell-2}))\\
&\le E_j{\cdot}\left\{ \binom{\ell-2+n-1}{n-1} + \binom{\ell-3+n-1}{n-1}\right\} + \ell_R (\Ext_R^{j} (M, S/Q^{\ell-3})) \\
&\le \cdots \\
& \le E_j{\cdot}\sum_{i=0}^{\ell-2}\binom{i+n-1}{n-1} = E_j{\cdot}\binom{\ell-2+n}{n},
\end{split}
\]
\end{small}
where the first equality follows from \eqref{nn244} and the first inequality follows from \eqref{xx2} for $i=\ell-1$ and the fact that $\ell_R(Y) =\ell_R(\Im f) +\ell_R(\Ker g) \le \ell_R(X) + \ell_R(Z)$ for an exact sequence $X\xrightarrow{f} Y \xrightarrow{g} Z$. The second, third, and fourth inequalities also follow in the same way as the first inequality by using \eqref{xx2} for $i=\ell-2, \dots, 2$ recursively. Therefore, by the above inequalities, we obtain that 
\[
E_{j+1}\le \frac{\binom{\ell+n-2}{n}}{\binom{\ell+n-2}{n-1}}{\cdot}{E_j}=\frac{\ell-1}{n}{\cdot}{E_j}
\]
for all $j>0$. Hence, for large enough integer $m\ge 0$, 
\[
E_{m+1}\le \left(\frac{\ell-1}{n}\right)^m E_1<1
\]
since $\ell\le n$. Hence $\Ext_R^j(M, S/Q)=0$ for all $j> m$. 
On the other hand, for $j>0$, $\Ext_R^{j+1}(M, S/Q)=0$ implies that $\Ext_R^{j} (M, S/Q^{\ell-1})=0$ by \eqref{nn244}. It follows that $\Ext_R^{j} (M, S/Q^{i})=0$ for all $1\le i \le \ell-1$ by looking at (\ref{xx2}) for $i=\ell-1, \ell-2, \dots, 2$ recursively. In conclusion, 
$\Ext_R^{j+1}(M, S/Q)=0$ implies that $\Ext_R^{j}(M, S/Q)=0$ for all $j>0$.
By using descending induction, $\Ext_R^j(M, S/Q)=0$ for all $j>0$.
\end{proof}

By Lemma \ref{ob1}(1) and Claim \ref{claim1}, we get $\Tor_{>0}^{R} (M, S/Q)=0$. Moreover, by applying the functor $M\otimes_R -$ to the sequences \eqref{nnn242}, we obtain
\begin{align}\label{xx3}
\Tor_{>0}^{R} (M, S/Q^i)=0 \ \ \text{ for all $1\le i \le \ell-1$}. 
\end{align}
Apply the functor $M\otimes_R -$ to \eqref{nnn242} again. Then we also get the exact sequence 
\begin{align}\label{xx4}
0\to (M/QM)^{\oplus \binom{i-1+n-1}{n-1}} \to M/Q^i M \to M/Q^{i-1} M \to 0 
\end{align}
of $R$-modules for all $2\le i \le \ell$ by (\ref{xx3}). Therefore, by applying the functor $\Hom_R (M, -)$ to (\ref{xx4}), we get the following exact sequence and isomorphism:
%\begin{small}
\begin{align*}
\cdots &\to \Ext_R^{j} (M, M/QM)^{\oplus \binom{i-1+n-1}{n-1}} \to \Ext_R^{j} (M, M/Q^{i}M) \to \Ext_R^{j} (M, M/Q^{i-1}M) \\
&\to \Ext_R^{j+1} (M, M/QM)^{\oplus \binom{i-1+n-1}{n-1}} \to \cdots 
\end{align*}
of $R$-modules for all $2\le i\le \ell-1$, $j>0$ and
\[
\Ext_R^{j} (M, M/Q^{\ell-1}M)\cong \Ext_R^{j+1} (M, M/QM)^{\oplus \binom{\ell-1+n-1}{n-1}} 
\]
%\end{small}
for all $j>0$. By setting $E'_j=\ell_R (\Ext_R^{j} (M, M/QM))$ for all $j>0$ and calculation like the proof of Claim \ref{claim1}, we have $\Ext_R^{>0}(M, M/QM)=0$. This induces that $\Ext_{S/Q}^{>0}(M/QM, M/QM)=0$ by (\ref{xx3}) and Lemma \ref{ob1} (2). 
Thus we have 
\[
\Ext_{S/Q}^{>0}(M/QM, M/QM\oplus S/Q)=0
\]
since $S/Q$ is Artinian Gorenstein, whence $M/QM$ is $S/Q$-free by Proposition \ref{e2.1}. This shows that $M$ is $R$-free since $\Tor_1^R (M, S/Q)=0$ by (\ref{xx3}).

$(3)\Rightarrow (1)$: Let $N$ be a finitely generated $S$-module and suppose that $\Ext_S^{>0}(N, N\oplus S)=0$. Then, by applying the functor $\Hom_S (N, -)$ to the exact sequence $0\to S\xrightarrow{x_1} S \to S/x_1 S \to 0$ of $S$-modules, $\Ext_S^{>0} (N, S/x_1 S)=0$. Hence 
\begin{align}\label{xx5}
\Ext_S^{>0} (N, S/Q)=0 
\end{align}
by induction on $n$. Similarly, $\Ext_S^{>0}(N, N/QN)=0$ since $N$ is a maximal Cohen-Macaulay $S$-module by $\Ext_S^{>0}(N, S)=0$.
Hence $\Tor_{>0}^{S} (N, S/Q)=0$ by Lemma \ref{ob1}(1). Moreover, by applying the functor $N\otimes_S -$ to the sequences \eqref{nnn242}, 
\begin{align}\label{xx6}
\Tor_{>0}^{S} (N, S/Q^i)=0\ \  \text{ for all $1\le i \le \ell$}. 
\end{align}
Therefore, we obtain from \eqref{nnn242} that
\begin{align}\label{xx7}
0\to (N/QN)^{\oplus \binom{i-1+n-1}{n-1}} \to N/Q^i N \to N/Q^{i-1} N \to 0 
\end{align}
is exact as $S$-modules for all $1\le i \le \ell$.
Hence, by applying the functor $\Hom_S (N, -)$ to (\ref{xx7}), 
$$\Ext_S^{>0} (N, N/Q^i N)=0\ \  \text{ for all $1\le i \le \ell$}$$
 since $\Ext_S^{>0}(N, N/QN)=0$. Thus $\Ext_{S/Q^\ell}^{>0} (N/Q^\ell N, N/Q^\ell N)=0$ by Lemma \ref{ob1}(2) and (\ref{xx6}). Similarly to the above argument, we have $\Ext_S^{>0} (N, S/Q^\ell)=0$ from \eqref{nnn242} and (\ref{xx5}). Whence $\Ext_{S/Q^\ell}^{>0} (N/Q^\ell N, S/Q^\ell\oplus N/Q^\ell N)=0$. Hence $N/Q^\ell N$ is $S/Q^\ell$-free, whence $N$ is $S$-free by (\ref{xx6}).
\end{proof}

The following assertions are direct consequences of Theorem \ref{a2.1}.

\begin{cor}\label{nn25}
Let $S$ be a Gorenstein complete local ring and $Q$ a parameter ideal of $S$ generated by a regular sequence on $S$.
Then the Auslander-Reiten conjecture holds for $S$ if and only if it holds for $S/Q^2$.

\end{cor}

\begin{cor}
Let $S$ be a Gorenstein complete local ring which is a complete intersection in codimension one. Let $x_1, x_2, \dots , x_n$ be a regular sequence on $S$ and set $Q=(x_1, x_2, \dots , x_n)$. 
Then the Auslander-Reiten conjecture holds for $S/Q^\ell$ for all $1\le \ell \le n$.
\end{cor}

\begin{cor}\label{cccc29}
Let $S$ be a local complete intersection (not necessary complete). Let $x_1, x_2, \dots , x_n$ be a regular sequence on $S$ and set $Q=(x_1, x_2, \dots , x_n)$. 
Then the Auslander-Reiten conjecture holds for $S/Q^\ell$ for all $1\le \ell \le n$.
\end{cor}

\begin{proof}
We note that $\widehat{S}$ is a complete intersection by \cite[Theorem 2.3.3]{BH}, thus the Auslander-Reiten conjecture holds for $\widehat{S}$ (\cite[4.2 Theorem]{AB}). Since $Q\widehat{S}$ is generated by a regular sequence of $\widehat{S}$, the conjecture holds for $\widehat{S}/Q^\ell\widehat{S}\cong \widehat{(S/Q^\ell)}$ for all $1\le \ell \le n$. It follows that the conjecture holds for $S/Q^\ell$ for all $1\le \ell \le n$ by Remark \ref{remrem23}.
\end{proof}

\begin{cor}\label{a2.2}
Let $R$ be a Cohen-Macaulay local ring. Suppose that there exists an isomorphism $R/\fkq\cong S/Q^2$ for some parameter ideal $\fkq$ of $R$, local complete intersection $S$ of dimension $r$, and parameter ideal $Q$ of $S$.
Then the Auslander-Reiten conjecture holds for $R$.
\end{cor}

\begin{proof}
The Auslander-Reiten conjecture holds for $S/Q^2$ by Corollary \ref{cccc29}. Since $R/\fkq\cong S/Q^2$, the conjecture holds for  $R/\fkq$. Hence, the Auslander-Reiten conjecture holds for $R$ by using Proposition \ref{e2.1} (2) $\Rightarrow$ (1) recursively.
\end{proof}

In Section \ref{section3}, we will characterize rings obtained in Corollary \ref{a2.2} by the existence of ideals in $R$. In the remainder of this section, we explore the Auslander-Reiten conjecture for determinantal rings. We start with the following.

\begin{prop}\label{f2.6}
Let $s$, $t$ be positive integers and assume that $2s\le t+1$. Let $\alpha_{ij}$ be positive integers for all $1\le i\le s$ and $1\le j\le t$.
Suppose that $S$ is a complete Gorenstein local ring and $\{ x_{ij} \}_{1\le i\le s, 1\le j\le t}$ forms a regular sequence on $S$.
Let $I=\mathbb{I}_s (x_{ij}^{\alpha_{ij}})$ be an ideal of $S$ generated by $s\times s$ minors of the $s\times t$ matrix $(x_{ij}^{\alpha_{ij}})$.
Set $R=S/I$. 
Then the Auslander-Reiten conjecture holds for $S$ if and only if it holds for $R$.

\end{prop}

\begin{proof}
First of all, we show the case where $\alpha_{ij}=1$ for all $1\le i\le s$ and $1\le j\le t$.
Set 
\begin{align*}
A&=\left\{ (i, j)\in \mathbb{Z}^2 \mid 1\le i\le s, 1\le j\le t \right\},\\
B&=\bigcup_{1\le i\le s-1}\left\{ (i, i+k)\in A \mid 0\le k\le t-s \right\},\text{ and}\\
C&=B\cup \left\{ (s, s+k)\in A \mid 0\le k\le t-s \right\}.
\end{align*}
Then 
\begin{align}\label{nn281}
\left\{ x_{ij}-x_{i+1 \,j+1} \mid (i, j)\in B \right\}\cup \left\{ x_{ij} \mid (i, j)\in A\setminus C \right\}
\end{align}
forms a regular sequence on $R$. 
In fact, letting $Q$ be an ideal of $S$ generated by the above sequence, we have 
\begin{align}\label{nn282}
R/QR\cong S/(I+Q)=S/[(x_{11}, x_{12}, \dots, x_{1\,t-s+1})^s+Q]
\end{align}
since we can transform the matrix $(x_{ij})$ into 
\begin{equation*}
\left(
\begin{array}{ccccccccc}
x_{11} &x_{12}&x_{13}&\cdots &x_{1\,t-s}&x_{1\,t-s+1}&&&\\
&x_{11}&x_{12}&\cdots&\cdots&x_{1\,t-s} &x_{1\,t-s+1}&&\\
&&\ddots&\ddots&\ddots&\ddots &\ddots&\ddots&\\
&&&x_{11}&x_{12}&x_{13}&\cdots&x_{1\,t-s} &x_{1\,t-s+1}\\
\end{array}
\right)
\end{equation*}
in the quotient ring $S/Q$. Hence, $R/QR$ is Artinian. On the other hand, we have $\dim R=\dim S-\height_S I\ge st-(t-s+1)$ by \cite[(2.1) Theorem]{BV2}. Since the length of the sequence \eqref{nn281} is of $st-(t-s+1)$, we obtain that $\dim R=st-(t-s+1)$ by the Krull's height theorem. Hence, $J$ is perfect in the sense of \cite[Definition 1.4.15]{BH} by \cite{EN} or \cite[(2.14) Proposition]{BV2}. Thus, $R$ is Cohen-Macaulay. It follows that the sequence \eqref{nn281} is a regular sequence on $R$. 

On the other hand, we note that the ideal generated by $\{x_{11}, x_{12}, \dots, x_{1\, t-s+1}\}$ and the sequence \eqref{nn281} is 
\[
(x_{ij} \mid 1\le i \le s, 1\le j \le t).
\] 
Hence, $\{x_{11}, x_{12}, \dots, x_{1\, t-s+1}\}$ is a regular sequence on $S/Q$. Therefore, the Auslander-Reiten conjecture holds for $S/Q$ if  and only if it holds for $S/[(x_{11}, x_{12}, \dots, x_{1\,t-s+1})^s+Q]$ by Theorem \ref{a2.1} and the assumption that $s\le t-s+1$.
By the assumption that the conjecture holds for $S$ and the isomorphism \eqref{nn282}, the conjecture holds for $R$ by Proposition \ref{e2.1}.

Let us consider the case where $\alpha_{ij}$ are arbitrary positive integers. Set $D=\left\{ (i, j)\in A \mid \alpha_{ij}>1\right\}$. We prove our assertion by induction on $N=\sharp D$.
Assume that $N>0$ and our assertion holds for $N-1$. Take $(i, j)\in D$. We may assume that $(i, j)=(s, 1)$.
Let $J$ be an ideal of $S$ generated by $s\times s$ minors of the $s\times t$ matrix $(x_{ij}^{\beta_{ij}})$, where $\beta_{ij}=\alpha_{ij}$ for all $(i, j)\in A\setminus \{ (s, 1)\}$ and $\beta_{s1}=1$. 
By the same way that we consider \eqref{nn282}, we can observe that $x_{s1}$ is a non-zerodivisor of $R$ and $S/J$. Therefore, since $R/x_{s1}R\cong S/(x_{s1}S+I) =S/(x_{s1}S+J)$, the Auslander-Reiten conjecture holds for $R$ if and only if that holds for $S/J$ by Proposition \ref{e2.1}. By the induction hypothesis, the conjecture holds for $S/J$, thus it holds for $R$. 
\end{proof}

\begin{cor}\label{ccc212}
On the assumption of $S$ in Proposition \ref{f2.6}, replace a complete Gorenstein local ring with a complete intersection. Then the Auslander-Reiten conjecture holds for $R=S/\mathbb{I}_s (x_{ij}^{\alpha_{ij}})$.
\end{cor}

\begin{proof}
We note that $\widehat{S}$ is a complete intersection by \cite[Theorem 2.3.3]{BH}, thus the Auslander-Reiten conjecture holds for $\widehat{S}$ (\cite[4.2 Theorem]{AB}). Hence, the Auslander-Reiten conjecture holds for $\widehat{R} \cong \widehat{S}/\mathbb{I}_s (x_{ij}^{\alpha_{ij}}) \widehat{S}$ by Proposition \ref{f2.6}, and so does $R$ by Remark \ref{remrem23}.
\end{proof}

%%%%%%%%%%%%%%%%%%%%%%%%%%%%%%%%%%%%%%%%%%%%%%%%%%%%%%%%%%%%
%%%%%%%%%%%%%%%%%%%%%%%%%%%%%%%%%%%%%%%%%%%%%%%%%%%%%%%%%%%%

The form obtained in Proposition \ref{f2.6} sometimes appears in several classes of rings.  

\begin{ex}
Let $A$ be a complete Gorenstein local ring for which is a complete intersection in codimension one. Let $A[[t]]$ be the formal power series ring over $A$. Then we have the following.
\begin{enumerate}[{\rm (1)}] 
\item (cf. \cite{B}) The Auslander-Reiten conjecture holds for all the Rees algebras $A[[Qt]]$ of parameter ideals $Q$.
\item (cf. \cite{H2}) The Auslander-Reiten conjecture holds for all three generated numerical semigroup rings $A[[t^a, t^b, t^c]]$, where $a, b, c\in \mathbb{N}$.
\end{enumerate}
\end{ex}

Furthermore, see \cite[Theorem 1.2]{GKMT}. Let us note another concrete example.

\begin{ex}
Suppose $A$ is a complete Gorenstein local ring for which the Auslander-Reiten conjecture holds. Let $n$ be a positive integer. Let $A[[t]]$ and $S=A[[X, Y, Z, W]]$ be formal power series rings over $A$.
Set $R=A[[t^{10}, t^{14}, t^{16}, t^{2n+1}]]$ and assume that $n \ge 6$. Then there exists an element $f\in (X)$ such that
\[
R\cong S/[\mathbb{I}_2(
\begin{smallmatrix}
X&Y^2&Z\\
Y&Z^2&X^2
\end{smallmatrix}
)+(W^2-f)],
\]
where $\mathbb{I}_2(\mathbb{M})$ denote the ideal of $S$ generated by $2\times 2$-minors of the matrix $\mathbb{M}$.
In particular, the Auslander-Reiten conjecture holds for $R$.
\end{ex}

\begin{proof}
Let $\varphi:S\to R$ be a homomorphism of $A$-algebras such that 
\begin{center}
$X\mapsto t^{10}$, $Y\mapsto t^{14}$, $Z\mapsto t^{16}$, and $W\mapsto t^{2n+1}$.
\end{center}
Then, one can check that $\Ker \varphi=\mathbb{I}_2(
\begin{smallmatrix}
X&Y^2&Z\\
Y&Z^2&X^2
\end{smallmatrix}
)+(W^2-f)$, where
\[
f={\footnotesize
\begin{cases}
X^m{\cdot}Y^{m-1}{\cdot}Z & \text{if $n=6m$}\\
X^{m+2}{\cdot}Z^{m-1} & \text{if $n=6m+1$}\\
X^{m+1}{\cdot}Y^{m} & \text{if $n=6m+2$}\\
X^{m}{\cdot}Y^{m-1} & \text{if $n=6m+3$}\\
X^{m}{\cdot}Y^{m-1}{\cdot}Z^{2} & \text{if $n=6m+4$}\\
X^{m+2}{\cdot}Y^{m-1}{\cdot}Z & \text{if $n=6m+5$}
\end{cases}}
\]
 for some positive integer $m$ by following the observation below. (Or, one can apply \cite[Proposition 5.2]{DKS} with regarding $\langle 10, 14, 16, 2n+1\rangle = 2\langle 5,7,8\rangle + \langle 2n+1\rangle$ and use the result of \cite{H2} to compute the kernel of the ring homomorphism $A[[X, Y, Z]] \to A[[t^{5}, t^{7}, t^{8}]]; X\mapsto t^{5}, Y\mapsto t^{7}, Z\mapsto t^{8}$.) 
\end{proof}

\begin{observation}
Suppose that there is a surjective ring homomorphism $\varphi: S\to R$ of Cohen-Macaulay local rings. Set $d=\dim R$. Let $I$ be an ideal of $S$ as a candidate for $\Ker \varphi$ and assume that $I\subseteq \Ker \varphi$. Then, one can check if $I=\Ker \varphi$ in the following way.

Consider an exact sequence $0 \to \Ker \varphi/I \to S/I \to R \to 0$. Let $\ol{x_1}, \dots, \ol{x_d}$ be a regular sequence on $R$, where $\ol{x}$ denotes the image of $x\in S$ in $R$. Set $Q=(x_1, \dots, x_d)$. By using the snake lemma recursively, we have 
\[
0 \to \Ker \varphi/(I + Q{\cdot}\Ker \varphi) \to S/(I+Q) \to R/QR \to 0.
\]
By noting that $R/QR$ is Artinian, $\ell_S(S/(I+Q)) = \ell_S(R/QR)$ if and only if $\Ker \varphi/(I +~Q{\cdot}\Ker \varphi)=0$. The latter condition is equivalent to saying that $I=\Ker \varphi$ by Nakayama's lemma. 
\end{observation}

%%%%%%%%%%%%%%%%%%%%%%%%%%%%%%%%%%%%%%%%%%%%%%%%%%%%%%%%%%%%
%%%%%%%%%%%%%%%%%%%%%%%%%%%%%%%%%%%%%%%%%%%%%%%%%%%%%%%%%%%%

Let us consider determinantal rings, which are not local rings.
From now on until the end of this section, let $s$, $t$ be positive integers. Let $A$ be a commutative ring and $A[\mathbf{X}]=A[X_{ij}]_{1\le i\le s, 1\le j\le t}$ be a polynomial ring over $A$. 
Suppose that $s\le t$ and $\mathbb{I}_s (\mathbf{X})$ is an ideal of $A[\mathbf{X}]$ generated by $s\times s$ minors of the $s\times t$ matrix $\mathbf{X}=(X_{ij})$.

\begin{lem}\label{h2.7}%{\rm \cite[Theorem 2.3.6.]{BH}}
With the above assumptions and notations, suppose that $A$ is a complete intersection for all localization at maximal ideals of $A$. Then $A[\mathbf{X}]$ is also a complete intersection  for all localizations at maximal ideals of $A[\mathbf{X}]$.
\end{lem}

\begin{proof}
Let $P\in \Max A[\mathbf{X}]$ and set $\fkp=P\cap A$. Then the natural map $A_\fkp\to A[\mathbf{X}]_P$ is a flat local homomorphism and its fiber $A[\mathbf{X}]_P/\fkp A_\fkp A[\mathbf{X}]_P\cong (A_\fkp/\fkp A_\fkp)[\mathbf{X}]_P$ is a regular local ring (see the proof of \cite[Theorem 23.5]{Mat}). Hence, since $A_\fkp$ is a complete intersection, $A[\mathbf{X}]_P$ is also a complete intersection by \cite{Av} (see also \cite[Remark 2.3.5]{BH}). 
\end{proof}

%\begin{lem}\label{h2.7}%{\rm \cite[Theorem 2.3.6.]{BH}}
%With the above assumptions and notations, suppose that $A$ is a Gorenstein ring which is a complete intersection in codimension one. Then $A[\mathbf{X}]$ is also a Gorenstein ring which is a complete intersection in codimension one.
%\end{lem}

%\begin{proof}
%For instance, see \cite[Theorem 2.3.6. and Corollary 3.3.21.]{BH}.
% and \cite[Theorem 4.5.3 and 4.5.5]{SH}. 
%Let $P\in \Spec A[\mathbf{X}]$ and set $\fkp=P\cap A$. Then the natural map $A_\fkp\to A[\mathbf{X}]_P$ is a flat local homomorphism and its fiber $A[\mathbf{X}]_P/\fkp A_\fkp A[\mathbf{X}]_P\cong (A_\fkp/\fkp A_\fkp)[\mathbf{X}]_P$ is a regular local ring (see the proof of \cite[Theorem 23.5]{Mat}). Hence, since $A_\fkp$ is Gorenstein, $A[\mathbf{X}]_P$ is also Gorenstein by \cite[Theorem 23.4]{Mat}. Suppose that $\height_{A[\mathbf{X}]} P\le 1$, that is, $\dim A[\mathbf{X}]_P\le 1$. Then 
%\[
%\height_A \fkp =\dim A_\fkp=\dim A[\mathbf{X}]_P - \dim (A_\fkp/\fkp A_\fkp)[\mathbf{X}]_P \le 1
%\] 
%by \cite[Theorem 15.1]{Mat}. Therefore, since $A_\fkp$ is a complete intersection, $A[\mathbf{X}]_P$ is also a complete intersection by \cite{Av} (see also \cite[Remark 2.3.5]{BH}).
%\end{proof}

\begin{thm}\label{f2.7}
Suppose that $A$ is a complete intersection for all localization at maximal ideals of $A$. Then the Auslander-Reiten conjecture holds for the determinantal ring $A[\mathbf{X}]/\mathbb{I}_s(\mathbf{X})$ if $2s\le t+1$.
\end{thm}

\begin{proof}
The case where $s=1$ is trivial. Suppose that $s>1$. Let $\fkN$ be a maximal ideal of $A[\mathbf{X}]$ such that $\fkN\supseteq \mathbb{I}_s(\mathbf{X})$. It is sufficient to show that the Auslander-Reiten conjecture holds for $\left(A[\mathbf{X}]/\mathbb{I}_s (\mathbf{X})\right)_\fkN$. Note that $A[\mathbf{X}]_\fkN$ is a complete intersection by Lemma \ref{h2.7}.
For integers $1\le p\le s$ and $1\le q\le t$, let
$$
\fkM_{pq}=(X_{ij}\mid 1\le i\le p, 1\le j\le q)
$$ 
denote a monomial ideal of $A[\mathbf{X}]$. 
The case where $\fkM_{st} \subseteq \fkN$ follows from Corollary \ref{ccc212}. 

%First, we consider the case where $\fkM_{st} \subseteq \fkN$. Then the Auslander-Reiten conjecture holds for $\widehat{A[\mathbf{X}]_\fkN}/ \mathbb{I}_s(\mathbf{X}) \widehat{A[\mathbf{X}]_\fkN} \cong \widehat{(A[\mathbf{X}]/\mathbb{I}_s(\mathbf{X}))_\fkN}$ by Proposition \ref{f2.6}. Thus, it holds for $(A[\mathbf{X}]/\mathbb{I}_s(\mathbf{X}))_\fkN$ by Remark \ref{remrem23}.

%The case where $\fkM_{st} \subseteq \fkN$ follows from Proposition \ref{f2.6}. 

Suppose that $\fkM_{st} \not\subseteq \fkN$ and take a variable $X_{ij}$ so that $X_{ij}\not\in\fkN$. We may assume that $X_{ij}=X_{st}$.
Then the matrix $\mathbf{X}=(X_{ij})$ is transformed to 
\begin{equation*}
\left(
\begin{array}{ccc|c}
&&&0 \\ 
&X_{ij}-\frac{X_{it}{\cdot}X_{sj}}{X_{st}}&&\vdots\\
&&&0 \\ \hline
0&\cdots&0& 1\\
\end{array}
\right)
\end{equation*}
by elementary transformation in $A[\mathbf{X}]_\fkN$. By \cite[(2.4) Proposition]{BV2}, we have the isomorphism $\varphi:A[\mathbf{X}][X_{st}^{-1}]\to A[\mathbf{X}][X_{st}^{-1}]$ of $A$-algebras, where 
\[\varphi(X_{ij})=
\begin{cases}
X_{ij}-\frac{X_{it}{\cdot}X_{sj}}{X_{st}}& \text{if $X_{ij}\in \fkM_{s-1\, t-1}$}\\
X_{ij}& \text{otherwise.}
\end{cases}
\]
Therefore, we have the commutative diagram

\[
\xymatrix{
A[\mathbf{X}]_\fkN \ar[r]^{\varphi_\fkN} \ar@{}[d]|{\rotatebox{90}{$\subseteq$}} \ar@{}[rd] |{\circlearrowleft}& A[\mathbf{X}]_\fkN \ar@{}[d]|{\rotatebox{90}{$\subseteq$}} \\
\mathbb{I}_s (\mathbf{X})_\fkN \ar[r] & \mathbb{I}_{s-1} (\mathbf{X}_{st})_\fkN, \\
}
\]
where $\mathbf{X}_{st}$ is the $(s-1)\times (t-1)$ matrix that results from deleting the $s$-th row and the $t$-th column of $\mathbf{X}$ and $\mathbb{I}_{s-1} (\mathbf{X}_{st})$ is an ideal of $A[\mathbf{X}]$ generated by $(s-1)\times (s-1)$ minors of $\mathbf{X}_{st}$. Hence 
$(A[\mathbf{X}]/\mathbb{I}_s (\mathbf{X}))_\fkN\cong (A[\mathbf{X}]/\mathbb{I}_{s-1} (\mathbf{X}_{st}))_\fkN$ as rings. 
If $\fkM_{s-1\,t-1} \subseteq \fkN$, the Auslander-Reiten conjecture holds for $(A[\mathbf{X}]/\mathbb{I}_{s-1} (\mathbf{X}_{st}))_\fkN$ by the same way of the case where $\fkM_{st} \subseteq \fkN$ since $2(s-1)\le (t-1)+1$. 
Assume that $\fkM_{s-1\,t-1} \not\subseteq \fkN$. Then, by repeating the above argument, we finally see that the Auslander-Reiten conjecture holds for $(A[\mathbf{X}]/\mathbb{I}_{s-1} (\mathbf{X}_{st}))_\fkN$ after finite steps.
\end{proof}

%%%%%%%%%%%%%%%%%%%%%%%%%%%%%%%%%%%%%%%%%%%%%%%%%%%%%%%%%%%%
%%%%%%%%%%%%%%%%%%%%%%%%%%%%%%%%%%%%%%%%%%%%%%%%%%%%%%%%%%%%
%%%%%%%%%%%%%%%%%%%%%%%%%%%%%%%%%%%%%%%%%%%%%%%%%%%%%%%%%%%%

\section{Ulrich ideals whose residue rings are complete intersections}\label{section3}

In this section, we study rings obtained in Corollary \ref{a2.2} in connection with the existence of ideals. 
Throughout this section, let $(R, \fkm)$ be a Cohen-Macaulay local ring of dimension $d$.

\begin{lem}
\label{a2.5}
Let $I$ be an $\fkm$-primary ideal of $R$ and $\fkq=(x_1, x_2, \dots, x_d)$ be a parameter ideal of $R$. Set $n = \mu_R (I)$. Suppose the following two conditions.
\begin{enumerate}[{\rm (1)}] 
\item $\fkq\subseteq I$ and $x_1, x_2, \dots, x_d$ is part of a minimal generating set of $I$.
\item $I^2\subseteq \fkq$ and $I/\fkq$ is $R/I$-free.
\end{enumerate}

Then $\rmr(R) = (n-d){\cdot}\rmr(R/I)$. In particular, $n=d+\rmr(R)$ if $R/I$ is a Gorenstein ring.

\end{lem}

\begin{proof}
Since $I/\fkq \cong (R/I)^{\oplus (n-d)}$, $I=\fkq:_R I$. Hence $I/\fkq = (\fkq:_R I)/\fkq\cong \Hom_R (R/I, R/\fkq)$. Therefore,
\begin{align*}
& \Hom_R (R/\fkm, (R/I)^{\oplus (n-d)}) \cong \Hom_R (R/\fkm, I/\fkq) \cong \Hom_R (R/\fkm, \Hom_R (R/I, R/\fkq)) \\
\cong & \Hom_R (R/\fkm\otimes_R R/I, R/\fkq) \cong \Hom_R (R/\fkm, R/\fkq).
\end{align*}
It follows that 
\begin{align*}
\rmr(R) =& \ell_R(\Ext_R^d (R/\fkm, R)) = \ell_R(\Hom_R (R/\fkm, R/\fkq)) = \ell_R(\Hom_R (R/\fkm, R/I))(n-d)\\
=& (n-d){\cdot}\rmr(R/I),
\end{align*}
where the second equality follows from \cite[Lemma 3.1.16]{BH}.
\end{proof}

%%%%%%%%%%%%%%%%%%%%%%%%%%%%%%%%%%%%%%%%%%%%%%%%%%%%%%%%%%%%

For a moment, let $(R, \fkm)$ be an Artinian local ring. Then there are a regular local ring $(S, \fkn)$ and a surjective local ring homomorphism $\varphi: S\to R$ (\cite[Theorem 29.4(ii)]{Mat}). We can choose $S$ so that the dimension of $S$ is equal to the embedding dimension of $R$. Set $v=\rmv(R)=\dim S$ and $r=\rmr(R)$. With these assumptions and notations, we have the following.

\begin{prop}%{\rm (Artinian case of Theorem \ref{a2.3})}
\label{a2.6}
Let $(R, \fkm)$ be an Artinian local ring and $S$ be as above.
The following conditions are equivalent.
\begin{enumerate}[{\rm (1)}] 
\item $r\le v$ and there exists a regular sequence $X_1, X_2, \dots, X_v\in \fkn$ on $S$ such that  
$$
R\cong S/\left[ (X_1, X_2, \dots, X_r)^2 + (X_{r+1}, X_{r+2}, \dots, X_v) \right]
$$
as rings.
\item There exists a nonzero ideal $I$ of $R$ such that 
\begin{enumerate}[{\rm (i)}] 
\item $I^2=0$ and $I$ is $R/I$-free.
\item $R/I$ is a complete intersection.
\end{enumerate}

\end{enumerate}

\end{prop}

\begin{proof}
(2) $\Rightarrow$ (1): Let $\ol{\varphi}: S\xrightarrow{\varphi} R\to R/I$ be a surjective local ring homomorphism. Set $\fka=\Ker \varphi$ and $J=\Ker \ol{\varphi}$ respectively. Since $R/I\cong S/J$ is a complete intersection, $J$ is generated by a regular sequence $x_1, x_2, \dots, x_v\in \fkn$ on $S$, see \cite[Theorem 2.3.3.(c)]{BH}. 
Hence, after renumbering of $x_1, x_2, \dots, x_v$, 
$$
I=J R=(\ol{x_1}, \ol{x_2}, \dots, \ol{x_v})=(\ol{x_1}, \ol{x_2}, \dots, \ol{x_r})
$$ 
by Lemma \ref{a2.5}, where $\ol{x}$ denotes the image of $x\in S$ in $R$. Thus $J=(x_1, x_2, \dots, x_r)+\fka$. For all $r+1\le i \le v$, take $y_i\in \fka$ and $c_{i_1}, c_{i_2}, \dots, c_{i_r}\in R$ so that $x_i=y_i + \sum_{j=1}^{r} c_{i_j} x_j $. Then $J=(x_1, x_2, \dots, x_r) + (y_{r+1}, y_{r+2}, \dots, y_v)$. Set 
$X=(x_1, x_2, \dots, x_r)$ and $Y=(y_{r+1}, y_{r+2}, \dots, y_v)$, where $Y$ is $(0)$ if $r=v$. We then have inclusions 
$$
J^2 + Y \subseteq \fka \subseteq J,
$$
where the first inclusion follows from $I^2=0$. On the other hand, setting $S'=S/Y$, 
$$
\ell_S (J/\left[ J^2 + Y \right])=\ell_{S'}(XS'/X^2 S')=\ell_{S'}(S'/XS'){\cdot}r=\ell_R (R/I){\cdot}r=\ell_R (I)=\ell_S (J/\fka),
$$
where the forth equality follows from the fact that $I$ is an $R/I$-free module. Thus $\fka=J^2 + Y=(x_1, x_2, \dots, x_r)^2 + (y_{r+1}, y_{r+2}, \dots, y_v)$.

(1) $\Rightarrow$ (2): Let $X=(X_1, X_2, \dots, X_r)$ and $Y=(X_{r+1}, X_{r+2}, \dots, X_v)$ be ideals of $S$.
Set $I=XR$. Then $I^2=0$ and $I\cong\left[ X+Y \right]/\left[ X^2+Y \right]$. Set $S'=S/Y$. Since $X$ is an ideal generated by a regular sequence on $S'$, by Lemma \ref{nnn24}, we obtain that 
\[
\left[ X+Y \right]/\left[ X^2+Y \right]\cong XS'/X^2S' \cong (S'/XS')^{\oplus r} \cong \left[ S/(X+Y) \right]^{\oplus r}.
\]
\end{proof}

%%%%%%%%%%%%%%%%%%%%%%%%%%%%%%%%%%%%%%%%%%%%%%%%%%%%%%%%%%%

Let us generalize Proposition \ref{a2.6} to arbitrary Cohen-Macaulay local rings.

\begin{thm}\label{a2.3}
Let $(R, \fkm)$ be a Cohen-Macaulay local ring. Then the following conditions are equivalent.
\begin{enumerate}[{\rm (1)}] 
\item There exists an isomorphism of rings $R/\fkq\cong S/Q^2$ for some ideal $\fkq$ of $R$ generated by a regular sequence on $R$, local complete intersection $S$, and ideal $Q$ of $S$ generated by a regular sequence on $S$.
\item There exists an isomorphism of rings $R/\fkq\cong S/Q^2$ for some parameter ideal $\fkq$ of $R$, local complete intersection $S$, and parameter ideal $Q$ of $S$.
\item There exists an isomorphism of rings $R/\fkq\cong S/Q^2$ for some parameter ideal $\fkq$ of $R$, local complete intersection $S$ of dimension $r$, and parameter ideal $Q$ of $S$.
\item There exist an $\fkm$-primary ideal $I$ and a parameter ideal $\fkq$ of $R$ such that
\begin{enumerate}[{\rm (i)}] 
\item $I^2\subseteq \fkq\subsetneq I$ and $I/\fkq$ is $R/I$-free.
\item $R/I$ is a complete intersection.
\end{enumerate}
\end{enumerate}

\end{thm}

\begin{proof}%[Proof of Theorem \ref{a2.3}]
(3) $\Rightarrow$ (2) and (2) $\Rightarrow$ (1) are clear. 

(1) $\Rightarrow$ (4): Set $\ell=\dim S$. We choose a regular sequene $x_1, x_2, \dots, x_s\in S$ so that $Q=(x_1, x_2, \dots, x_s)$ for some $s\le \ell$. 
Since $R/\fkq\cong S/Q^2$, $S/Q^2$ is a Cohen-Macaulay local ring and the dimension is $\ell-s$. Hence, we can choose  a regular sequene $x_{s+1}, x_{s+2}, \dots, x_\ell\in S$ on $S/Q^2$. Set $Q'=(x_{s+1}, x_{s+2}, \dots, x_\ell)$. 
By considering the isomorphism $R/\fkq\cong S/Q^2$, we can choose $y_1, y_2, \dots, y_\ell \in R$ so that $\ol{y_i}$ corresponds to $\ol{x_i}$ for all $1\le i\le \ell$, where $\ol{x_i}$ denotes the image of $x_i$ in $S/Q^2$ and $\ol{y_i}$ denotes the image of $y_i$ in $R/\fkq$. 
Set 
\begin{center}
$\fkq'=\fkq + (y_{s+1}, y_{s+2}, \dots, y_\ell)$ and $I= \fkq' + (y_1, y_2, \dots, y_s)$. 
\end{center}
Then $R/\fkq'\cong (R/\fkq)/(y_{s+1}, y_{s+2}, \dots, y_\ell)(R/\fkq) \cong (S/Q^2)/ Q'(S/Q^2)\cong S/[Q^2+Q']$; hence, $\fkq'$ is a parameter ideal on $R$ and $I$ is an $\fkm$-primary ideal of $R$. Furthermore, we obtain that 
\[
I^2=\fkq'I+(y_1, y_2, \dots, y_s)^2 \subseteq \fkq'\subseteq I
\]
since $Q=(x_1, x_2, \dots, x_s)$ and $R/\fkq'\cong S/[Q^2+Q']$. Again, by the isomorphism $R/\fkq'\cong S/[Q^2+Q']$, we also have 
\[
I/\fkq'\cong [Q+Q']/[Q^2+Q']\cong Q(S/Q')/Q^2(S/Q')\cong [S/(Q+Q')]^{\oplus s} \cong (R/I)^{\oplus s},
\]
where the third isomorpshim follows by Lemma \ref{nnn24}. We can also observe that $R/I$ is a complete intersection by the isomorphism $R/I\cong S/(Q+Q')$.

%(1) $\Rightarrow$ (3): Set $\ell=\dim S$ and $Q=(x_1, x_2, \dots, x_\ell)$. Then, since $R/\fkq\cong S/Q^2$, we can choose $y_1, y_2, \dots, y_\ell \in R$ so that $\ol{y_i}$ corresponds to $\ol{x_i}$ for all $1\le i\le \ell$, where $\ol{x_i}$ denotes the image of $x_i$ in $S/Q^2$ and $\ol{y_i}$ denotes the image of $y_i$ in $R/\fkq$. Set $I=(y_1, y_2, \dots, y_\ell) + \fkq$. Then $R/I\cong S/Q$ is a complete intersection and $I/\fkq\cong Q/Q^2$ is $R/I$-free. %Note that $Q/Q^2$ is nonzero since $S$ is a positive dimension.  Furthermore $I^2=\fkq I + (y_1, y_2, \dots, y_\ell)^2\subseteq \fkq$.

(4) $\Rightarrow$ (3): This follows from the observation that $I/\fkq$ is an ideal of $R/\fkq$ which satisfies the assumption of Proposition \ref{a2.6}(2).
\end{proof}

%%%%%%%%%%%%%%%%%%%%%%%%%%%%%%%%%%%%%%%%%%%%%%%%%%%%%%%%%%%%

Theorem \ref{a2.3} is applicable to Ulrich ideals. Here the definition of Ulrich ideals is stated as follows.

\begin{defn}(\cite[Definition 2.1.]{GOTWY})\label{a2.4}
Let $(R, \fkm)$ be a Cohen-Macaulay local ring and $I$ an $\fkm$-primary ideal of $R$. Assume that $I$ contains a parameter ideal $\fkq$ of $R$ as a reduction. We say that $I$ is an {\it Ulrich ideal} of $R$ if the following conditions are satisfied.
\begin{enumerate}[{\rm (1)}] 
\item $I\not=\fkq$, but $I^2=\fkq I$. 
\item $I/I^2$ is a free $R/I$-module.
\end{enumerate}

\end{defn}

Note that the condition (1) of Definition \ref{a2.4} is independent of the choice of a reduction $\fkq$; see, for example,  \cite[Theorem 2.1.]{H}. The following assertions claim that the notion of Ulrich ideals is closely related to the condition (3)(i) of Theorem \ref{a2.3}.

\begin{prop}{\rm (\cite[Lemma 2.3. and Proposition 2.3.]{GOTWY})}\label{b2.8}
\begin{enumerate}[{\rm (1)}] 
\item If $I$ is an Ulrich ideal of $R$, then $I^2=\fkq I\subseteq \fkq$ and $I/\fkq$ is $R/I$-free for every parameter ideal $\fkq$ of $R$ such that $\fkq$ is a reduction of $I$.  
\item Assume that $R/\fkm$ is infinite.  If $I^2\subseteq \fkq$ and $I/\fkq$ is $R/I$-free for all minimal reductions $\fkq$ of $I$, then $I$ is an Ulrich ideal of $R$.
\end{enumerate}

\end{prop}

%%%%%%%%%%%%%%%%%%%%%%%%%%%%%%%%%%%%%%%%%%%%%%%%%%%%%%%%%%%%

We are in a position to prove the main result of this section. For convenience, set $d=\dim R$, $r=\rmr(R)$, and $v=\rmv(R)$.

\begin{thm}\label{a2.7}
Suppose that there are a regular local ring $(T, \fkn)$ of dimension $v$ and a surjective local ring homomorphism $\varphi: T\to R$. If there exists an Ulrich ideal $I$ of $R$ such that $R/I$ is a complete intersection, then $\mu_R (I)=d+r\le v$ and there exists a regular sequence $x_1, x_2, \dots, x_v$ on $T$ such that 
\begin{enumerate}[{\rm (1)}] 
\item $\ol{x_1}, \ol{x_2}, \dots, \ol{x_d}$ is a regular sequence on $R$, where $\ol{x}$ denotes the image of $x\in T$ in $R$.
\item $R/(x_1, x_2, \dots, x_d)R\cong T/\left[ (x_1,\dots, x_d)+ (x_{d+1}, \dots, x_{d+r})^2 + (x_{d+r+1}, \dots, x_v) \right]$.
\end{enumerate}
Therefore, letting 
$0\to F_{v-d}\to \cdots \to F_1\to F_0 \to R \to 0$
be a minimal $T$-free resolution of $R$, we can compute the Betti numbers as follows.
$$
\rank_T F_0=1\ \  \text{and}\ \  \rank_T F_i=\sum_{j=0}^{v-r-d} \beta_{i-j}{\cdot}\binom{v-r-d}{j}
$$
for $1\le i\le v-d$, where $\beta_k=
\begin{cases}
1 & \text{if $k=0$}\\
k{\cdot}\binom{r+1}{k+1} & \text{if $1\le k\le r$}\\
0& \text{otherwise.}
\end{cases}$

In particular, $\rank_T F_i=i{\cdot}\binom{r+1}{i+1}$ for $1\le i\le r$ if $\mu_R (I)=v$.

\end{thm}

\begin{proof}
Let $\ol{\varphi}: T\xrightarrow{\varphi} R\to R/I$ be a surjective local ring homomorphism. Set $\fka=\Ker \varphi$ and $J=\Ker \ol{\varphi}$. Since $R/I\cong T/J$ is a complete intersection, $J$ is generated by a regular sequence $x_1, x_2, \dots, x_v\in \fkn$ on $T$. 
Hence, after renumbering of $x_1, x_2, \dots, x_v$, 
$$
I=J R=(\ol{x_1}, \ol{x_2}, \dots, \ol{x_v})=(\ol{x_1}, \ol{x_2}, \dots, \ol{x_{d+r}})
$$ 
by Lemma \ref{a2.5}, where $\ol{x}$ denotes the image of $x\in T$ in $R$. Let $(\ol{x_1'}, \ol{x_2'}, \dots, \ol{x_{d}'})\subseteq I$ be a minimal reduction of $I$. Then, after renumbering of $x_1, x_2, \dots, x_{d+r}$, 
\[
I=(\ol{x_1'}, \ol{x_2'}, \dots, \ol{x_{d}'}, \ol{x_{d+1}}, \dots, \ol{x_{d+r}}).
\] 
Thus $J=(x_1', \dots, x_d')+(x_{d+1}, \dots, x_{d+r})+\fka$. Since $\mu_T (J)=v$, we can choose $v$ elements in $\{x_1', \dots, x_d', x_{d+1}, \dots, x_{d+r}\} \cup \{a \mid a\in \fka\}$ as a minimal system of generators.
Assume that $x_i'$ cannot be chosen as a part of minimal system of generators. Then 
\[
I=(\ol{x_1'}, \dots, \ol{x_{i-1}'}, \ol{x_{i+1}'}, \cdots, \ol{x_{d}'}, \ol{x_{d+1}}, \dots, \ol{x_{d+r}}).
\]
This is contradiction for $\mu_R (I)=r+d$ by Lemma \ref{a2.5}. Hence 
$$
J=(x_1', \dots, x_d')+(x_{d+1}, \dots, x_{d+r})+(y_{d+r+1}, \dots, y_v)
$$ 
for some $y_{d+r+1}, \dots, y_v\in \fka$. Set $X_1=(x_1', \dots, x_d')$, $X_2=(x_{d+1}, \dots, x_{d+r})$, and $Y=(y_{d+r+1}, \dots, y_v)$. Then $X_1 + X_2^2 + Y\subseteq \fka + X_1\subseteq J$, whence $\fka + X_1=X_1 + X_2^2 + Y$ since $\ell_T (J/[\fka+X_1])=\ell_T (J/[X_1 + X_2^2+Y])=r{\cdot}\ell_T (T/J)$.

Let $0\to F_{v-d}\to \cdots \to F_1\to F_0 \to R \to 0$ be a minimal $T$-free resolution of $R$. Then 
\[
0\to F_{v-d}/X_1 F_{v-d}\to \cdots \to F_1/X_1 F_1\to F_0/X_1 F_0 \to R/X_1 R \to 0
\] 
is a minimal $T/X_1$-free resolution of $R/X_1 R$ and $R/X_1 R\cong T/[X_1 + X_2^2 + Y]$ since $\fka + X_1=X_1 + X_2^2 + Y$.
On the other hand, the Eagon-Northcott complex \cite{EN} gives the minimal $T/X_1$-free resolution 
\[
0\to G_{r}\xrightarrow{\partial_r} G_{r-1}\xrightarrow{\partial_{r-1}} \cdots \xrightarrow{\partial_{2}} G_{1}\xrightarrow{\partial_{1}} G_{0}\xrightarrow{\partial_{0}} T/[X_1 + X_2^2]\to 0
\]
of $T/[X_1 + X_2^2]$, where $\rank_{T/X_1}G_k=\beta_k$ for all $k\in \mathbb{Z}$.
Therefore, considering the mapping cone inductively, we have $\rank_T F_i=\sum_{j=0}^{v-r-d} \beta_{i-j}{\cdot}\binom{v-r-d}{j}$ as desired.
\end{proof}

Theorem \ref{a2.7} recovers the result of J. Sally \cite[Theorem 1(iii)]{S} by taking the maximal ideal $\fkm$ as a certain Ulrich ideal. We note that \cite[Theorem 1(i) and (ii)]{S} are known in more generality. Indeed, the Hilbert function of an $\fkm$-primary ideal $I$ such that $I^2=QI$ for some parameter ideal $Q$ is $\ell_R(R/I^{n+1}) = \ell_R(R/I)\binom{n+d}{d} - \ell_R(I/Q)\binom{n+d-1}{n+d-1}$ for all $n\ge 0$. This gives a generalization of \cite[Theorem 1(i)]{S} from the maximal ideal to an $\fkm$-primary ideal $I$ such that $I^2=QI$ for some parameter ideal $Q$. A generalization of \cite[Theorem 1(ii)]{S} from the maximal ideal to an Ulrich ideal are in \cite[Theorem 7.1]{GOTWY}. 

\begin{cor}{\rm (\cite[Theorem 1(iii)]{S})}
Let $(R, \fkm)$ be a $d$-dimensional Cohen-Macaulay local ring of multiplicity $e$ and embedding dimension $v=e+d-1$. If $R=T/\fka$ with $(T, \fkn)$ a regular local ring and $\fka\subseteq \fkn^2$, then 
\begin{center}
$\Tor^T_i (R, R/\fkm)=0$ for $i>e-1$ \quad and \quad $\dim \Tor^T_i (R, R/\fkm)=i\binom{e}{i+1}$ for $i=1, \dots, e-1$.
\end{center}
\end{cor}

\begin{proof}
Since $v-d=e-1$ by the assumption, we have  a minimal $T$-free resolution $0\to F_{e-1}\to \cdots \to F_1\to F_0 \to R \to 0$ of $R$. By noting that $\rank_T F_i = \dim \Tor^T_i (R, R/\fkm)$, we obtain the assertion by Theorem \ref{a2.7}.
\end{proof}

%%%%%%%%%%%%%%%%%%%%%%%%%%%%%%%%%%%%%%%%%%%%%%%%%%%%%%%%%%%%

\begin{rem}
With the assumptions of Theorem \ref{a2.7}, the equality $\mu_R (I)=v$ does not necessarily hold in general; see Example \ref{f3.11} (1). On the other hand, if $R$ is a one-dimensional Cohen-Macaulay local ring possessing minimal multiplicity, then for every Ulrich ideal $I$ one has $R/I$ is a complete intersection and $\mu_R (I)=v$; see \cite[Theorem 4.5]{GIK}.

\end{rem}

Combining Theorems \ref{a2.1}, \ref{a2.3}, and \ref{a2.7}, we have the following which is a goal of this paper.

\begin{cor}\label{c2.10}
Let $(R, \fkm)$ be a Cohen-Macaulay local ring of dimension $d$. Suppose that there exists an Ulrich ideal whose residue ring is a complete intersection. 
Then the following assertions are true.
\begin{enumerate}[{\rm (1)}] 
\item The Auslander-Reiten conjecture holds for $R$.
\item $r+d\le v$.
\item There is an isomorphism of rings $R/\fkq\cong S/Q^2$ for some parameter ideal $\fkq$ of $R$, local complete intersection $S$ of dimension $r$, and parameter ideal $Q$ of $S$.
\item Suppose that there are a regular local ring $T$ of dimension $v$ and a surjective ring homomorphism $T\to R$.  Let
$0\to F_{v-d}\to \cdots \to F_1\to F_0 \to R \to 0$
be a minimal $T$-free resolution of $R$. Then
$$
\rank_T F_0=1\ \  \text{and}\ \  \rank_T F_i=\sum_{j=0}^{v-r-d} \beta_{i-j}{\cdot}\binom{v-r-d}{j}
$$
for $1\le i\le v-d$, where $\beta_k=
\begin{cases}
1 & \text{if $k=0$}\\
k{\cdot}\binom{r+1}{k+1} & \text{if $1\le k\le r$}\\
0& \text{otherwise.}
\end{cases}$
\end{enumerate}
\end{cor}

\begin{proof}
(1): This follows from (3) and Corollary \ref{a2.2}.
%Let $I$ be an Ulrich ideal of $R$ whose residue ring is a complete intersection. Then so is the ideal $I\widehat{R}$ of $\widehat{R}$, where $\widehat{R}$ denotes the completion of $R$, by noting that $R\to \widehat{R}$ is a flat local homomorphism of rings and by checking the definition of Ulrich ideals. In addtion, $\widehat{R}$ contains a field. Thus, to conclude the assertion, we may assume that $R$ is complete. Since $R$ is complete and $R$ contains a field, $R$ contains a complete regular local ring $A$ over which $R$ is a finitely generated (\cite[Theorem 29.4(iii)]{Mat}). Hence, there exists a surjective ring homomorphism $A[[X_1, \dots, X_\ell]] \to R$ for some $\ell>0$ and $A[[X_1, \dots, X_\ell]]$ is a complete regular local ring. By applying Theorem \ref{a2.7}(2) with regarding $T=A[[X_1, \dots, X_\ell]]$, the assertion follows from Corollary \ref{a2.2}.

(2): Passing to the completion of $R$, we may assume that there exist a regular local ring $T$ of dimension $v$ and a surjective ring homomorphism $T\to R$. Then the assertion follows from Theorem \ref{a2.7}.

(3): This follows from Theorem \ref{a2.3} and Proposition \ref{b2.8}.

(4): This is Theorem \ref{a2.7}.
\end{proof}

Below we construct many examples of Ulrich ideals whose residue ring is a complete intersection.

%Note that it is not necessarily unique for a given ring that an Ulrich ideal whose residue ring is a complete intersection.

\begin{prop}%{\rm (cf. \cite{GIK})}
\label{b2.10}
Let $(S, \fkn)$ be a local complete intersection of dimension three and $f, g, h\in \fkn$ a regular sequence on $S$. 
Set 
%$$
%R=S/
%\mathbb{I}_2\left(\begin{smallmatrix}
%f&g &h\\
%g&h&f\\
%\end{smallmatrix}
%\right),
%$$
%where $\mathbb{I}_2\left(\mathbb{M}\right)$ denotes the ideal generated by $2\times 2$-minors of the matrix $\mathbb{M}$. 
$$
R=S/(f^2-gh, g^2-hf, h^2-fg).
$$
Then $R$ is a Cohen-Macaulay local ring of dimension one and $I=(f, g, h)R$ is an Ulrich ideal of $R$ such that $R/I$ is a complete intersection. Furthermore, if $f=f_1{\cdot}f_2$ for $f_1, f_2\in\fkn$, then $I_1=(f_1, g, h)R$ is also an Ulrich ideal of $R$ such that $R/I_1$ is a complete intersection.

\end{prop}

\begin{proof}
By direct calculation, we see
\begin{center}
$I^2=fI$, $\ell_R (R/I)=\ell_S (S/(f, g, h))$, and $\ell_R (I/fR)=\ell_S ((f, g, h)/[(f)+(g, h)^2])=2{\cdot}\ell_S (S/(f, g, h))$.
\end{center}
Hence a surjection $(R/I)^{\oplus 2} \to I/fR$ must be an isomorphism, that is, $I$ is an Ulrich ideal of $R$ and $R/I\cong S/(f, g, h)$ is a complete intersection.

Assume that $f=f_1{\cdot}f_2$. Then, we also have
\begin{center}
$I_1^2=f_1 I_1$, $\ell_R (R/I_1)=\ell_S (S/(f_1, g, h))$, and $\ell_R (I_1/f_1 R)=\ell_S ((f_1, g, h)/[(f_1)+(g, h)^2])=2{\cdot}\ell_S (S/(f_1, g, h))$.
\end{center}
Hence $I_1$ is an Ulrich ideal of $R$ and $R/I_1\cong S/(f_1, g, h)$ is a complete intersection.
\end{proof}

Here are some examples arising from Proposition \ref{b2.10}.

\begin{ex}
With the same notations of Proposition \ref{b2.10}, assume that $S=k[[X, Y, Z]]$ is a formal power series ring over a field $k$. Let $\ell, m, n$ be positive integers such that $(\ell, m, n)\ne (0, 0, 0)$.
Then we have the following examples. 
\begin{enumerate}[{\rm (1)}] 
\item If $f, g, h$ in Proposition \ref{b2.10} are $X^\ell, Y^m, Z^n$ respectively, then 
\begin{center}
$(X^i, Y^m, Z^n)R$, $(X^\ell, Y^j, Z^n)R$, $(X^\ell, Y^m, Z^k)R$
\end{center}
are Ulrich ideals for all $0\le i\le \ell$, $0\le j\le m$, $0\le k\le n$.
\item If $f, g, h$ in Proposition \ref{b2.10} are $X^\ell{\cdot}Y^m{\cdot}Z^n, X^{2}+Y^{2}, Y^{2}+Z^{2}$ respectively, then 
\begin{center}
$(X^i{\cdot}Y^j{\cdot}Z^k, X^{2}+Y^{2}, Y^{2}+Z^{2})R$
\end{center}
are Ulrich ideals for all $0\le i\le \ell$, $0\le j\le m$, $0\le k\le n$.
Furthermore $(X^\ell{\cdot}Y^m{\cdot}Z^n, X+Y, Y^{2}+Z^{2})R$ is also an Ulrich ideal if $k$ is a field of characteristic two or an algebraically closed field.
\end{enumerate}

\end{ex}

We close this paper with several examples.

\begin{ex}\label{f3.11}
Let $k[[t]]$ and $S=k[[X, Y, Z, W]]$ be formal power series rings over a field $k$. Then the following assertions are true.
\begin{enumerate}[{\rm (1)}] 
\item Set $R_1=k[[t^6, t^{11}, t^{16}, t^{26}]]$, and let $I=(t^6, t^{16}, t^{26})$ be an ideal of $R_1$.
Then $I$ is an Ulrich ideal of $R_1$ with a reduction $(t^6)$, and $R_1/I$ is a complete intersection. 
Therefore, the minimal $S$-free resolution of $R_1$ has the following form
\[
0\to S^{\oplus 2}\to S^{\oplus 5}\to S^{\oplus 4}\to S \to R_1 \to 0.
\]

On the other hand, 
$R_1\cong S/\left( X^7-ZW, Y^2-XZ, Z^2-XW, W^2-X^6Z \right)$ as rings, thus $R_1$ does not have the form obtained in Proposition \ref{f2.6}.
\item Set $R_2=S/(X^2-YZ, Y^2-ZX, Z^2-XY, W^2)$.
%$$
%R=S/\left[
%\mathbb{I}_2\left(\begin{smallmatrix}
%X&Y &Z\\
%Y&Z&X\\
%\end{smallmatrix}
%\right)+(W^2)\right],
%$$
%where $\mathbb{I}_2\left(\mathbb{M}\right)$ denotes the ideal generated by $2\times 2$-minors of the matrix $\mathbb{M}$. 
Then $X$ is a non-zerodivisor of $R_2$ and $R_2/XR_2\cong k[[Y, Z, W]]/\left[ (Y, Z)^2 + (W^2) \right]$. Hence the Auslander-Reiten conjecture holds for $R_2$. On the other hand, $I=(X, W)R_2$ is an Ulrich ideal, whence $I$ is a non-free totally reflexive $R_2$-module (\cite[Theorem 2.8.]{GTT2}). Hence $R_2$ is not $G$-regular in the sense of \cite{Tak}. In particular, $R_2$ is neither an almost Gorenstein ring nor a Golod ring.
\end{enumerate}
\end{ex}

%%%%%%%%%%%%%%%%%%%%%%%%%%%%%%%%%%%%%%%%%%%%%%%%%%%%%%%%%%%
%%%%%%%%%%%%%%%%%%%%%%%%%%%%%%%%%%%%%%%%%%%%%%%%%%%%%%%%%%%%
%%%%%%%%%%%%%%%%%%%%%%%%%%%%%%%%%%%%%%%%%%%%%%%%%%%%%%%%%%%%%%

\begin{acknowledgments}
The author is grateful to Ryo Takahashi for kind suggestions about the manuscript, which he gave the author while preparing for the final version. He also pointed out an error in Proposition \ref{e2.1}. The author also thanks to the anonymous referee for his/her numerous comments improving the readability of this paper. Furthermore, the condition (1) of Theorem 3.3 is suggested by the referee.
\end{acknowledgments}

%%%%%%%%%%%%%%%%%%%%%%%%%%%%%%%%%%%%%%%%%%%%%%%%%%%%%%%%%%%%%%%%%%%%%%%%%%%%%%%%

%\addcontentsline{toc}{section}{references}


\begin{thebibliography}{BB}

\bibitem{Ara}
{\sc T. Araya}, The Auslander-Reiten conjecture for Gorenstein rings, {\em Proceedings of the American Mathematical Society}, {\bf 137} (2009), no.6, 1941--1944.

\bibitem{ADS}
{\sc M. Auslander, S. Ding, and \O. Solberg}, Liftings and Weak Liftings of Modules, {\em Journal of Algebra}, {\bf 156} (1993), 273--317.

\bibitem{AR}
{\sc M. Auslander, I. Reiten}, On a generalized version of the Nakayama conjecture, {\em Proceedings of the American Mathematical Society}, {\bf 52} (1975), 69--74.

\bibitem{Av}
{\sc L. L. Avramov}, Flat morphisms of complete intersections. (Russian) {\em Dokl. Akad. Nauk SSSR}, {\bf 225} (1975), no. 1, 11--14.

\bibitem{AB}
{\sc L. L. Avramov, R. O. Buchweitz}, Support varieties and cohomology over complete intersections. {\em Invent. Math.} {\bf 142} (2000), no. 2, 285--318.

%\bibitem{ABS}
%{\sc L. L. Avramov, R. O. Buchweitz, L. M. \c{S}ega}, Extensions of a dualizing complex by its ring:commutative versions of a conjecture of Tachikawa, {\em Journal of Pure and Applied Algebra}, {\bf 201} (2005), 218--239.

\bibitem{B}
{\sc J. Barshay}, Graded algebras of powers of ideals generated by A-sequences, {\em Journal of Algebra}, {\bf 25} (1973), 90--99.

%\bibitem{BHU}
%{\sc J. P. Brennan, J. Herzog, and B. Ulrich}, Maximally generated maximal Cohen-Macaulay modules, {\em Mathematica Scandinavica}, {\bf 61} (1987), 181--203.

\bibitem{BH}
{\sc W. Bruns and J. Herzog}, Cohen-Macaulay Rings, Cambridge University Press (1993).

%\bibitem{BV}
%{\sc J. P. Brennan and W. V. Vasconcelos}, On the structure of closed ideals, {\em Mathematica Scandinavica}, {\bf 88} (2001), 3--16.


\bibitem{BV2}
{\sc W. Bruns and U. Vetter}, Determinantal Rings, Lecture Notes in Mathematics 1327, (1988), {\em Springer-Verlag, Berlin}.

\bibitem{CT}
{\sc O. Celikbas and R. Takahashi}, Auslander-Reiten conjecture and Auslander-Reiten duality, {\em Journal of Algebra}, {\bf 382} (2013), 100--114.


%\bibitem{CGKM}
%{\sc T. D. M. Chau, S. Goto, S. Kumashiro, and N. Matsuoka}, Sally modules of canonical ideals in dimension one and $2$-AGL rings, {\em Journal of Algebra}, {\bf 521} (2019), 299--330.


\bibitem{CH}
{\sc L. W. Christensen and H. Holm}, Algebras that satisfy Auslander's condition on vanishing of cohomology, {\em Mathematische Zeitschrift}, {\bf 265} (2010), no. 1, 21--40.



%\bibitem{CP}
%{\sc A. Corso and C. Polini}, Links of prime ideals and their Rees algebras, {\em J. Algebra}, {\bf 178} (1995), no. 1, 224--238.

%\bibitem{DEL}
%{\sc H. Dao, M. Eghbali, and J. Lyle}, Hom and Ext, revisited, {\em Journal of Algebra}, (to appear), arXiv:1710.05123.

\bibitem{DKS}
{\sc S. Dey, S. Kumashiro, and P. Sarkar}, On a generalized Auslander-Reiten conjecture, arXiv:2209.12718.

\bibitem{EN}
{\sc J. A. Eagon and D. G. Northcott}, Ideals defined by matrices and a certain complex associated with them, {\em Proceedings of the Royal Society A}, {\bf 269} (1962), 188--204.



%\bibitem{FW}
%{\sc S. A. Seyed Fakhari and V. Welker}, The Golod property for products and high powers of monomial ideals, {\em Journal of Algebra}, {\bf 400} (2014), 290--298.




%\bibitem{GGHV}
%{\sc L. Ghezzi, S. Goto, J. Hong, and W. V. Vasconcelos}, Invariants of Cohen-Macaulay rings associated to their canonical ideals, {\em J. Algebra} (to appear).

\bibitem{GKMT}
{\sc S. Goto, D. V. Kien, N. Matsuoka, and H. L. Truong}, Pseudo-Frobenius numbers versus defining ideals in numerical semigroup rings, {\em Journal of Algebra}, {\bf 508} (2018), 1--15.


\bibitem{GIK}
{\sc S. Goto, R. Isobe, and S. Kumashiro}, The structure of chains of Ulrich ideals in Cohen-Macaulay local rings of dimension one, {\em Acta Mathematica Vietnamica}, {\bf 44}(1) (2019), 65--82.


%\bibitem{GK2}
%{\sc S. Goto and S. Kumashiro}, When is $R \ltimes I$ an almost Gorenstein ring?, Proc. Amer. Math. Soc., {\bf 146} (2018), 1431--1437.

%\bibitem{GK}
%{\sc S. Goto and S. Kumashiro}, On generalized Gorenstein local rings, (in preparation).


%\bibitem{GKL}
%{\sc S. Goto, S. Kumashiro, and N. T. H. Loan}, Residually faithful modules and the Cohen-Macaulay type of idealizations, {\em Journal of the Mathematical Society of Japan}, (to appear), arXiv:1804.07885.



%\bibitem{GMP}
%{\sc S. Goto, N. Matsuoka, T.T. Phuong}, Almost Gorenstein rings, {\em J. Algebra}, {\bf 379} (2013), 355--381.


\bibitem{GOTWY}
{\sc S. Goto, K. Ozeki, R. Takahashi, K.-i. Watanabe, and K.-i. Yoshida}, Ulrich ideals and modules, {\em Mathematical Proceedings of the Cambridge Philosophical Society}, {\bf 156} (2014), 137--166.


\bibitem{GOTWY2}
{\sc S. Goto, K. Ozeki, R. Takahashi, K.-i. Watanabe, and K.-i. Yoshida}, Ulrich ideals and modules over two-dimensional rational singularities, {\em Nagoya Mathematical Journal}, {\bf 221} (2016), 69--110.


%\bibitem{GT}
%{\sc S. Goto and R. Takahashi}, On the Auslander-Reiten conjecture for Cohen-Macaulay local rings, {\em Proceedings of the American Mathematical Society}, {\bf 145} (2017), 3289--3296.







\bibitem{GTT}
{\sc S. Goto, R. Takahashi and N. Taniguchi}, Almost Gorenstein rings - towards a theory of higher dimension, {\em Journal of Pure and Applied Algebra}, {\bf 219} (2015), 2666--2712.



\bibitem{GTT2}
{\sc S. Goto, R. Takahashi, and N. Taniguchi}, Ulrich ideals and almost Gorenstein rings, {\em Proceedings of the American Mathematical Society}, {\bf 144} (2016), 2811--2823.


%\bibitem{HK}
%{\sc J. Herzog and E. Kunz}, Der kanonische Modul eines-Cohen-Macaulay-Rings, Lecture Notes in Mathematics, {\bf 238}, Springer-Verlag, 1971.

%\bibitem{HH}
%{\sc D. Hanes and C. Huneke}, Some criteria for the Gorenstein property, {\em Journal of Pure and Applied Algebra}, {\bf 201} (2005), 4--16.

\bibitem{Ha}
{\sc D. Happel}, Homological conjectures in representation theory of finite dimensional algbras, {\em Sherbrook Lecture Notes Series}, (1991).


\bibitem{H2}
{\sc J. Herzog}, Generators and relations of Abelian semigroups and semigroup rings, {\em Manuscripta Mathematica}, {\bf 3} (1970), 175--193.

%\bibitem{HHS}
%{\sc J. Herzog, T. Hibi, and D. I. Stamate}, The trace ideal of the canonical module, arXiv:1612.02723v2.

\bibitem{H}
{\sc C. Huneke}, Hilbert functions and symbolic powers, {\em Michigan Mathematical Journal}, {\bf 34}  (1987), 293-318.


%\bibitem{HJ}
%{\sc C. Huneke and D. A. Jorgensen}, Symmetry in the vanising of Ext over Gorenstein rings, {\em Mathematica Scandinavica}, {\bf 93} (2003), 161--184.

\bibitem{HL}
{\sc C. Huneke and G. J. Leuschke}, On a conjecture of Auslander and Reiten,  {\em Journal of Algebra}, {\bf 275} (2004), 781--790.


\bibitem{HSV}
{\sc C. Huneke, L. M.  \c{S}ega, and A. N. Vraciu}, Vanishing of Ext and Tor over some Cohen-Macaulay local rings, {\em Illinois Journal of Mathematics}, {\bf 48} (2004), 295--317.



%%%%%%%%%%%%%%%%%%%%%%%%%%%%%%%%



\bibitem{JS}
{\sc D. A. Jorgensen and L. M. \c{S}ega}, Nonvanishing cohomology and classes of Gorenstein rings, {\em Advances in Mathematics}, {\bf 188}(2) (2004), 470--490.




%\bibitem{Kob}
%{\sc T. Kobayashi}, On delta invariants and indices of ideals, {\em Journal of the Mathematical Society of Japan}, (to appear).



%\bibitem{K}
%{\sc S. Kumashiro}, The Auslander-Reiten conjecture and residually faithful modules, arXiv:????????.


%\bibitem{KT}
%{\sc T. Kobayashi and R. Takahashi}, Ulrich modules over Cohen-Macaulay local rings with minimal multiplicity, arXiv:1711.00652.

%%%%%%%%%%%%%%%%%%%%%%%%%%%%%%%%



%\bibitem{Lin}
%{\sc H. Lindo}, Trace ideals and center of endomorphism rings of modules over commutative rings, {\em Journal of Algebra}, {\bf 482} (2017), 102--130.

%\bibitem{L}
%{\sc H. Lindo}, Self-injective commutative rings have no nontrivial rigid ideals, arXiv:1710.01793v2.

%\bibitem{Lipman}
%{\sc J. Lipman}, Stable ideals and Arf rings, {\em Amer. J. Math.}, {\bf 93} (1971), 649--685.

\bibitem{Mat}
{\sc H. Matsumura}, Commutative ring theory. Translated from the Japanese by M. Reid. Second edition. Cambridge Studies in Advanced Mathematics, 8. Cambridge University Press, Cambridge, (1989). 


%\bibitem{N}
%{\sc D. G. Northcott}, A note on the coefficients of the abstract Hilbert function, {\em Journal of the London Mathematical Society}, {\bf 35}  (1960), 209-214.

%\bibitem{O}
%{\sc A. Ooishi}, $\Delta$-genera and sectional gerena of commutative rings, {\em Hiroshima Mathematical Journal}, {\bf 17}  (1987), 361-372.

%\bibitem{R}
%{\sc I. Reiten}, The converse of a theorem of Sharp on Gorenstein modules, {\em Proc. Amer. Math. Soc.}, {\bf 32}  (1972), 417-420.

\bibitem{S} 
{\sc J. Sally}, Cohen-Macaulay local rings of maximal embedding dimension, {\em Journal of Algebra}, {\bf 56} (1979), 168--183.

%\bibitem{SH}
%{\sc I. Swanson and C. Huneke}, Integral closure of Ideals, Rings, and Modules, {\em London Mathematical Society Lecture Notes Series 336}, Cambridge University Press, (2006). 


\bibitem{Tak}
{\sc R. Takahashi}, On $G$-regular local rings, {\em Communications in Algebra}, {\bf 36}(12) (2008), 4472--4491.

\bibitem{W}
{\sc J. Wei}, Generalized Auslander-Reiten conjecture and tilting equivalences, {\em Proceedings of the American Mathematical Society}, {\bf 138}(5) (2010), 1581--1585.


\bibitem{W2}
{\sc J. Wei}, Auslander bounds and homological conjectures, {\em Revista Mathematica Iberoamericana}, {\bf 27}(3) (2011), 871--884.


\end{thebibliography}
\end{document}